\documentclass[11pt]{article}

\pdfoutput = 1

\usepackage{amsmath,amsfonts,amssymb}

\usepackage[hmargin=0.12\paperwidth,vmargin=0.12\paperwidth,bindingoffset=0cm]{geometry}

\usepackage{amsxtra}
\usepackage{color}

\usepackage{array,dcolumn}
\usepackage{graphicx}

\usepackage[hyperfootnotes=true,
        pdffitwindow=true,
        plainpages=false,
        pdfpagelabels=true,
        pdfpagemode=UseOutlines,
        pdfpagelayout=SinglePage,
        hyperindex,]{hyperref}

\numberwithin{equation}{section} 

\usepackage{amsthm}

\theoremstyle{plain}
\newtheorem{theorem}{Theorem}[section]
\newtheorem{lemma}[theorem]{Lemma}
\newtheorem{corollary}[theorem]{Corollary}
\newtheorem{proposition}[theorem]{Proposition}

\theoremstyle{definition}
\newtheorem{definition}[theorem]{Definition}

\newtheorem{remark}[theorem]{Remark}

\numberwithin{equation}{section}

\newcommand{\diag}{{\rm diag\,}}
\newcommand{\eins}{\leavevmode\hbox{\small1\kern-3.8pt\normalsize1}}

\DeclareMathOperator{\Tr}{Tr} 

\DeclareMathOperator{\gO}{O}

\begin{document}

\begin{center}
{\bfseries \LARGE  Multiplicative Convolution of Real  Asymmetric \\[0.4ex] 
and Real Antisymmetric Matrices}\\[2\baselineskip]
{\LARGE Mario Kieburg\footnote{mkieburg@physik.uni-bielefeld.de}}\\[.5\baselineskip]
{\itshape Fakult\"at f\"ur Physik,
Universit\"at Bielefeld, PO Box 100131, 33501 Bielefeld, Germany}
\\[\baselineskip]
{\LARGE P. J. Forrester\footnote{pjforr@unimelb.edu.au}, %
J.R. Ipsen\footnote{jesper.ipsen@unimelb.edu.au}}\\[.5\baselineskip]
{\itshape ARC Centre of Excellence for Mathematical and Statistical Frontiers,\\
School of Mathematics and Statistics, The University of Melbourne, Victoria 3010, Australia.}\\

\end{center}

\begin{abstract}
\noindent
The singular values of products of standard complex Gaussian random matrices, or
sub-blocks of Haar distributed unitary matrices, have the property that their probability
distribution has an explicit, structured form referred to as a polynomial ensemble.
It is furthermore the case that the corresponding bi-orthogonal system can be determined in
terms of Meijer G-functions, and the correlation kernel given as an explicit double contour
integral. It has recently been shown that the Hermitised product
$X_M \cdots X_2 X_1A X_1^T X_2^T \cdots X_M^T$, where each
$X_i$ is a standard real complex Gaussian matrix, and $A$ is real anti-symmetric shares exhibits
analogous properties. Here we use the theory of spherical functions and transforms to present
a theory which, for even dimensions, includes these properties of the latter product as a special case. 
As an example we show that the theory also allows for a
treatment of this class of Hermitised product when the $X_i$ are chosen as sub-blocks of
Haar distributed real orthogonal matrices.

\medskip
\noindent
{\bf Keywords:} products of random matrices; spherical function; bi-orthogonal functions; matrix integrals

\medskip
\noindent
{\bf MSC:} 15A52, 42C05
\end{abstract}

\section{Introduction}\label{sec:intro}

A theoretical  and practical feature of random matrix theory is that there are
many explicit, structured formulas for statistical quantities which hold for general values of
the matrix size. As an example, let $X$ be an $n \times n$ complex Gaussian matrix, or an
$n \times n$ sub-block of an $N \times N$ ($N \ge 2n$) unitary matrix drawn from the Haar measure.
In both cases there is a simple, explicit formula for the joint probability density function (jPDF) of the
squared singular values (or equivalently, eigenvalues of $X^\dagger X$) of the form
\begin{equation}\label{1}
\prod_{l=1}^n w(x_l) \Delta_n^2(x), \qquad
w(x) = \left \{ \begin{array}{ll} e^{-x} \Theta(x), & {\rm Gaussian}, \\
x^\mu(1 - x)^\nu \Theta(x(1-x)), & {\rm Jacobi\ (truncated \, unitary)}. \end{array} \right.
\end{equation}
Here $\Delta_n(x) := \prod_{1 \le j < k \le n}(x_k - x_j)$ is the Vandermonde product of differences,
$\Theta(y)$ is the
Heaviside step function, and
in the truncated unitary case the exponents $(\mu,\nu)$ depend on $n,N$. Moreover, an analogous expression
holds true in the case that $X$ is a real Gaussian matrix, or a sub-block of a real orthogonal matrix
chosen with Haar measure; the essential structural difference is that the Vandermonde product
$\Delta_n(x)$ occurs to the first power rather than to the second. See e.g.~\cite{Mehta,Fo10} for textbook
treatment of these and related results.

Consider now the product $X_1 X_2 \cdots X_M$ of random matrices where each $X_i$ is chosen
independently as an $n \times n$ complex Gaussian matrix, or chosen independently as an
$n \times n$ sub-block of a Haar distributed unitary matrix. It is a relatively recent finding
(see \cite{AIK13,AKW13,KKS15} or \cite{AI15} for a review) that the squared singular values form a so-called polynomial ensemble~\cite{KS14},
meaning that they have a jPDF of the form
\begin{equation}\label{2}
\Delta_n(x) \det [ g_j(x_k) ]_{j,k=1}^n
\end{equation}
for certain functions $\{g_j(x) \}_{j=1}^n$. It was furthermore shown that for both classes of $X_i$ the corresponding
bi-orthogonal system determining the correlation kernel can be made explicit by the use
of Meijer G-functions~\cite{AKW13,AIK13,KZ14,KKS15}. However, in distinction to the case
$M=1$, there are no known analogous results for general $M$ when each
$X_i$ is chosen as a
 real Gaussian matrix, or as a sub-block of a Haar distributed real orthogonal matrix.
 
 One explanation for the distinction between the cases of real and complex entries is the group integral
\begin{equation}\label{3} 
 \int_{K} \exp[ {\rm Tr} \,A k B k^\dagger] \, d^* k,
\end{equation} 
 where $K = {\rm O}(n)$ (real case) or $K = {\rm U}(n)$ (complex case). The crucial point is whether the matrices $A,B$ are in the Lie-algebra of the corresponding groups or in their ``dual" symmetric spaces. In the latter case $A,B$ are $n \times n$ real symmetric matrices or Hermitian matrices while in the former case they are real anti-symmetric or anti-Hermitian, respectively. When both matrices are in the corresponding Lie-algebra the group integral is known as the Harish-Chandra integral and permits an evaluation in terms of determinants~\cite{HC57}. The case of $A$ and $B$ in the symmetric space is known as the Itzykson--Zuber integral~\cite{Itzykson:1980} and an explicit, compact evaluation for arbitrary dimension $n$ is only possible for the unitary case $K ={\rm U}(n)$ because it equals with the Harish-Chandra integral. This is no longer true for the real case; see \cite{KKS15} for an extended discussion of the role of group integrals in the computation of the jPDF of the singular values for product matrices of Gaussians or truncated unitaries. Another, closely related, viewpoint has emerged from the recent joint work of K\"osters with one of the authors~\cite{KK16a,KK16b} which is based on harmonic analysis~\cite{He00}. Here, generalising the role played by
the Mellin transform in the study or products of scalar random variables, a theory of random product
matrices based on the spherical transform has been proposed. This in turn requires multi-dimensional
spherical functions, and it is only in the complex case that these functions admit explicit, structured
(determinantal) evaluations known as the Gelfand-Na\v{\i}mark integral~\cite{GN50}.

Let us return to the squared singular values of the random product matrix $X_1 X_2 \cdots X_M$. Those singular values are equal to
the eigenvalues of the Hermitised product
\begin{equation}\label{4} 
X_M \cdots X_2 X_1 X_1^\dagger X_2^\dagger \cdots X_M^\dagger.
\end{equation}
In a recent work involving two of the present authors \cite{FILZ17} it has been shown that taking
a viewpoint of Hermitised products does allow for a natural theory of explicit, structured formulas for
eigenvalue jPDFs and associated correlation functions for the $X_i$ being real Gaussian matrices. For this
to be possible, Eq.~(\ref{4}) must first be generalise to the form
\begin{equation}\label{4a} 
X_M \cdots X_2 X_1 A X_1^T X_2^T \cdots X_M^T,
\end{equation}
where $A$ is a real anti-symmetric matrix. The explicit results of~\cite{FILZ17} are
based on the Harish-Chandra group integral for the orthogonal group~\cite{HC57,Ey07},
\begin{equation}\label{5}
\int_{K=\gO(2n)} \exp\left[\frac{1}{2}\Tr(XkYk^T)\right]\, d^*k=  \prod_{k=0}^{n-1}(2k)! \frac{\det[\cosh x_iy_j]_{i,j=1,\ldots,n}}{\Delta _n(x^2)\Delta_n(y^2)},
\end{equation}
where $X,Y$ are $2n \times 2n$ antisymmetric matrices with
singular values $\{x_j\}, \, \{y_j\}$,
and its analogue for the odd dimensional case.

With regard to the harmonic analysis approach~\cite{KK16a,KK16b}, the 
natural question of an understanding of explicit formulas in relation to the Hermitised product~(\ref{4a}) from the viewpoint of spherical transforms arises. The purpose of the present work is to provide such
an insight in the case that the matrices are all even dimensional. As examples we do not only review the case when all $X_i$ in Eq.~(\ref{4a}) are real Gaussians but also extend this study to the case of truncations of real orthogonal matrices or more general matrices drawn from a real Jacobi ensemble.

Section~\ref{sec:spher}  specifies the appropriate spherical functions and spherical
transforms as relevant to Eq.~(\ref{4a}). For this purpose we introduce our notations and review the main ideas of the harmonic analysis approach of~\cite{KK16a,KK16b}. Theorem~\ref{thm:spher.func} specifies in terms of
determinants and Vandermonde products the central object of this study, namely the corresponding 
spherical function $\Phi(s;x)$. It is defined as the ratio of two group integrals, see Eq.~(\ref{spher:as}). The evaluation
of $\Phi(s;x)$, given in the Appendix, makes essential use of the Harish-Chandra integral~(\ref{5}) for the orthogonal group $K={\rm O}(2n)$.

To obtain analytically tractable expressions for the spherical transform and thus for the jPDF of the singular values, we require that $A$ in Eq.~(\ref{4a}) is either a fixed real anti-symmetric matrix or drawn from a polynomial ensemble, and that the $X_i$ are identically and
independently distributed according to what we term a factorising ensemble (see Definition~\ref{def:poly.ens}). We derive the jPDFs and bi-orthogonal functions of the product matrix~\eqref{4a} in Sec.~\ref{sec:gen}. In both situation the jPDFs satisfy the form of a polynomial ensemble, see Corollary~\ref{cor:jPDF.poly}  and Theorem~\ref{thm:jPDF.fixed}, and the corresponding bi-orthogonal system and the correlation kernels look very similar to results obtained for products of complex matrices~\cite{KK16b} as well as sums of Hermitian, real anti-symmetric, Hermitian anti-selfdual or complex rectangular matrices~\cite{Ki17}.

In Section~\ref{sec:examples} the general fomalism is specialised to the case that the
$X_i$ in Eq.~(\ref{4a}) are either real Gaussians, so reclaiming the corresponding results in \cite{FILZ17}, or truncations
of real orthogonal matrices (Jacobi ensemble). In particular the case of the jPDF of the squared singular values of $A=\imath\eins_n\otimes\tau_2$ ($\tau_2$ is the second Pauli matrix) surprisingly coincides with the jPDF of the eigenvalues of
$  Y_{2M} \cdots Y_1  Y_1^\dagger \cdots Y_{2M}^\dagger$, where each $Y_i$ is a particular
truncation of a Haar distributed complex unitary matrix, see Proposition~\ref{prop4.5}. This observation compliments the one in~\cite{FILZ17} made for the Gaussian case.

In Section~\ref{sec:conclusio} we summarize our results and outline some open problems related to the present work.

\section{Spherical Functions and Spherical Transforms}\label{sec:spher}

We study the multiplicative group action of the real general linear group ${\rm Gl}_{\mathbb{R}}(2n)$ on the even dimensional real antisymmetric matrices which is the Lie algebra ${\rm o}(2n)$ of the orthogonal group ${\rm O}(2n)$. To state our results, the following measurable matrix spaces
and associated notations are required:
\begin{enumerate}
\item	the real general linear group $G={\rm Gl}_{\mathbb{R}}(2n)$ equipped with the  Lebesgue measure $dg$,
\item	the real antisymmetric matrices $H={\rm o}(2n)$ equipped with the flat Lebesgue measure $dx$,
\item	the positive real diagonal $n\times n$ matrices $A\simeq\mathbb{R}_+^n$ equipped with the flat Lebesque measure $da$,
\item	the $2\times2$ block-diagonal real matrices $Z=\{z\in{\rm Gl}_{\mathbb{R}}(2n)| z=\diag(z_1,\ldots,z_n)\ {\rm with}\ z_j\in{\rm Gl}_{\mathbb{R}}(2)\}$ equipped with the flat Lebesgue measure $dz$,
\item	the group of the lower triangular matrices with a $2\times2$ block structure, unit elements on the diagonal
 $$T=\biggl\{t\in{\rm Gl}_{\mathbb{R}}(2n)\biggl| t=\{t_{ab}\}_{a,b=1\ldots,n}\ {\rm with}\ t_{ij}\in{\rm gl}_{\mathbb{R}}(2),\ t_{ij}=\left\{ \begin{array}{cl} 0, & i<j, \\ \eins_2, & i=j \end{array}\right.\qquad\biggl\}$$
			equipped with the flat Lebesgue measure $dt$,
\item	the orthogonal group $K={\rm O}(2n)$ equipped with the normalized Haar measure $d^*k$.
\end{enumerate}
We remark that singular matrices form a set of measure zero with respect to the flat Lebesgue measure.
The flat Lebesgue is the product of the differentials of all independent matrix entries. Additionally, we need the set of $K$-invariant Lebesgue integrable functions on $G$ and $H$  and the symmetric integrable functions on $A$. These we define as
\begin{equation}\label{L1.def}
\begin{split}
L^{1,K}(G)=&\{ f_G\in L^{1}(G)|\ f_G(g)=f_G(k_1gk_2)\ {\rm for\ all}\ g\in G\ {\rm and}\ k_1,k_2\in K\},\\
L^{1,K}(H)=&\{ f_H\in L^{1}(H)|\ f_H(x)=f_H(kxk^T)\ {\rm for\ all}\ x\in H\ {\rm and}\ k\in K\},\\
L^{1,\mathbb{S}}(A)=&\{f_A\in L^{1}(A)|\ f_A(a)=f_A(\sigma a\sigma^T)\ {\rm for\ all}\ a\in A\ {\rm and}\ \sigma\in \mathbb{S}\}.
\end{split}
\end{equation}
The matrix $k^T$ is the transposition of $k$ and the set $\mathbb{S}$ is the symmetric group of $n$ elements. We denote the subset of probability densities of these function spaces by $L_{\rm Prob}^{1,K}(G)$, $L_{\rm Prob}^{1,K}(H)$ and $L_{\rm Prob}^{1,K}(H)$. We adapt the notation of~\cite{KK16b} and 
include the space on which a function belongs as a subscript, like $f_G$, $f_H$ or $f_A$.

We equip the spaces $\mathcal{I}:L^{1,K}(H)$ and $L^{1,\mathbb{S}}(A)$ with the $L^1$-norm. Then there is an isometry $\mathcal{I}:L^{1,K}(H)\rightarrow L^{1,\mathbb{S}}(A)$ of the form
\begin{equation}\label{isometry}
f_A(a)=\mathcal{I}f_H(a)=C\Delta_n^2(a^2)f_H\left(\imath a\otimes\tau_2\right),\ C=\frac{1}{n!}\left(\prod_{j=0}^{n-1}\frac{2(2\pi)^{2j}}{(2j)!}\right),
\end{equation}
with the Vandermonde determinant $\Delta_n(a^2)=\prod_{1\leq k<l\leq n}(a_l^2-a_k^2)$ and $\tau_2$ the second Pauli matrix.  The formula (\ref{isometry}) comes
about from the change of variables associated with the decomposition
$H = K (\imath a \otimes \tau_2) K^T$; see e.g.~\cite[Eq.~(2.8)]{FILZ17}.

There is another map which is important in our calculations. It corresponds to the change of
variables associated with the particular Schur decomposition of a matrix $g\in G$ into the form $g=ktzk^T$ with $z\in Z$, $t\in T$ and $k\in K$. The jPDF $p_Z$ of $z$ for a given $K$-invariant distribution $P_G\in L_{\rm Prob}^{1,K}(G)$ is given by~\cite{Ed97}
\begin{equation}\label{jPDF.Z}
p_Z(z)= C_*\prod_{1\leq k<l\leq n}|\det(z_l\otimes\eins_2-\eins_2\otimes z_k)|\left(\prod_{j=1}^n \det (z_jz_j^T)^{n-j}\right)\int_T dt P_G(tz),
\end{equation}
where the explicit form of the constant $C_* = {\rm vol} \, {\rm O}(2n) / (n! [{\rm vol} \, {\rm O}(2)]^n)$ is not  important for our purposes.
The map we need in our results corresponds to  the final two factors in Eq.~\eqref{jPDF.Z},
\begin{equation}\label{T.def}
\tilde{p}_Z(z)=\mathcal{T}P_G(z)=\left(\prod_{j=1}^n \det (z_jz_j^T)^{n-j}\right)\int_T dt P_G(tz).
\end{equation}
Two remarks are in order. First, it is unclear whether this map can be inverted. For the counterpart of bi-unitary invariant complex matrices this map is indeed invertible, see~\cite{KK16a}. Second, it is also unclear whether  there is always a $K$-invariant probability distribution on $G$ for a given $K$-invariant probability distribution $p_Z$ on $Z$. In the case of complex matrices we have seen that this is not necessarily guaranteed.

The question which we want to address is the following: given two K-invariant random matrices $g\in G$ and $x\in H$ generated by the distributions $P_G\in L_{\rm Prob}^{1,K}(G)$ and $P_H\in L_{\rm Prob}^{1,K}(H)$, respectively, what is the jPDF of the singular values of  the matrix $y=gxg^T$? We follow the same general ideas as in~\cite{KK16b} and try to find  spherical transformations $\mathcal{S}_{\Phi}$ and $\mathcal{S}_{\Psi}$ on $H$ and $G$, respectively, such that the natural convolution on the two matrix spaces $G$ and $H$, given by
\begin{equation}\label{con:def}
P_G\circledast P_H(y)=\int_G \frac{dg}{(\det gg^T)^{(2n-1)/2}} P_G(g)P_H(g^{-1}y(g^{-1})^T)\in L_{\rm Prob}^{1,K}(H),
\end{equation}
satisfies the factorization
\begin{equation}\label{factor.conv}
\mathcal{S}_{\Phi}[P_G\circledast P_H](s)=\mathcal{S}_{\Psi}P_G(s)\mathcal{S}_{\Phi}P_H(s).
\end{equation}
The metric $dg/(\det gg^T)^{(2n-1)/2}$ has been chosen for the convolution and not the Haar measure $d[g]/(\det gg^T)^{n}$ on $G={\rm Gl}_{\mathbb{R}}(2n)$ because we want to have the natural normalization property
\begin{equation}\label{norm:conv}
\int_H dy P_G\circledast P_H(y)=\int_G dg P_G(g)\int_H dx P_H(x);
\end{equation}
note that for $a \in H$ fixed, $d(aya^T) = (\det a a^T)^{(2n-1)/2}dy$, see e.g.~\cite[Exercises 1.3 q.2]{Fo10}.

When the intermediate goal~\eqref{factor.conv} is accomplished we can perform the matrix averages separately and invert the spherical transform $\mathcal{S}_{\Phi}$ at the end. However the spherical transforms of arbitrary distributions $P_G$ and $P_H$ can have very complicated expressions. As we already know from the multiplicative convolution of complex matrices~\cite{KK16a} and from the additive convolution~\cite{KR16,FKK17} of Hermitian antisymmetric matrices, Hermitian matrices, Hermitian anti-self-dual matrices, and complex rectangular matrices one has to restrict the class of ensembles for the
results to be analytically tractable. For this reason we define the two sets of functions
\begin{equation}\label{L1.def.b}
\begin{split}
L^1_n(\mathbb{R}_+)=&\biggl\{f\in L^1(\mathbb{R}_+)\biggl|\int_0^\infty a^{s-1}|f(a)|da<\infty\ {\rm for\ all}\ s\in[1,2n-1]\biggl\}, \\
L^1_n({\rm Gl}_{\mathbb{R}}(2))=&\biggl\{h\in L^1({\rm Gl}_{\mathbb{R}}(2))\biggl|\int_{{\rm Gl}_{\mathbb{R}}(2)} |p(z)f(a)|da<\infty\ \text{ for\ all\ polynomials\ $p$\ in $z$\ homogenous\ of}\\
&\text{\ order}\ 0,\ldots,2n-2\biggl\}.
\end{split}
\end{equation}

\newpage
\begin{definition}[Polynomial Ensembles on $H$ and Factorizing Ensembles on $G$]\label{def:poly.ens}\
\begin{enumerate}
 \item	A polynomial ensemble on $H$ is a random matrix ensemble on $H$ with a probability distribution $P_H\in L_{\rm Prob}^{1,K}(H)$ such that the jPDF of the singular values $a\in A$ has the form
 			\begin{equation}\label{poly.def}
 			p_A(a)=\mathcal{I}P_G(a)=C_n[w]\Delta_n(a^2)\det[w_b(a_c)]_{b,c=1,\ldots,n}.
 			\end{equation}
 We call this a polynomial ensemble \cite{KS14} on $H$ associated with the weights $w_1,\ldots,w_n\in L^1_n(\mathbb{R}_+)$. The factor $C_n[w]$ is the normalization constant.
 \item	A factorizing ensemble on $G$ is a random matrix ensemble on $G$ with a probability density $P_G\in L_{\rm Prob}^{1,K}(G)$ such that
 			\begin{equation}\label{fact.def}
 			\tilde{p}_Z(z)=\mathcal{T}P_G(z)=\prod_{j=1}^n \sigma(z_j).
 			\end{equation}
			We call this a factorizing ensemble on $G$ associated with the weight $\sigma\in L^1_n({\rm Gl}_{\mathbb{R}}(2))$.
\end{enumerate}
\end{definition}

The definition of polynomial ensembles on $H$ slightly differs from~\cite{KS14} since we consider now the jPDF of the singular values of a random matrix $y\in H$ and not of its squared singular values.

The construction of the spherical transform satisfying the relation~\eqref{factor.conv} proceeds via spherical functions, see~\cite{He00,Te88a,FK94,JL01}. In the simplest cases these are the Fourier factors for the Fourier transform, the renormalized Bessel functions of the first kind for the Hankel transform or  the monomials of the Mellin transform. For the multiplication of complex random matrices a non-trivial generalization of the monomials have proven quite helpful~\cite{KK16a,KK16b}. Progress in the present case also requires the  identification of appropriate spherical functions.

\begin{definition}[Spherical Functions on $H$ and $G$]\label{def:spherical-functions}\

Let $x\in H$ and $g\in G$ be two fixed matrices and $s=\diag(s_1,\ldots,s_n)\in\mathbb{C}^n$ an $n$-tuple of complex numbers. Moreover we set $s_{n+1}=-n-1$ and specify the rectangular matrix $\Pi_{2j,2l}$ as the orthogonal projection from $2l$ rows to the first $2j$ rows. We define
\begin{enumerate}
\item	the spherical function on $H={\rm o}(2n)$ as
			\begin{equation}\label{spher:as}
				\Phi(s; x)=\frac{\int_{K}d^*k\prod_{j=1}^{n} [\det \Pi_{2j,2n}kxk^T\Pi_{2j,2n}^T]^{(s_j-s_{j+1})/2-1}}{\int_{K}d^*k\prod_{j=1}^{n} [\det \Pi_{2j,2n}k(\imath \eins_n\otimes\tau_2 )k^T\Pi_{2j,2n}^T]^{(s_j-s_{j+1})/2-1}}
			\end{equation}
			for ${\rm Re}\,(s_j-s_{j+1})\geq 2$ for all $j=1,\ldots,n-1$ and analytically continue $\Phi$ to ${\rm Re}\,(s_b-s_b+1)< 2$ for some $b=1,\ldots, n$
\item	and the spherical function on $G={\rm Gl}_{\mathbb{R}}(2n)$ by
			\begin{equation}\label{spher:s}
				\Psi(s; g)=\int_{K}d^*k\prod_{j=1}^{n} [\det \Pi_{2j,2n}kgg^Tk^T\Pi_{2j,2n}^T]^{(s_j-s_{j+1})/2-1}
			\end{equation}
			for all $s\in\mathbb{C}^n$.
\end{enumerate}
\end{definition}

The analytic continuation of $\Phi$ in $s$ as noted below Eq.~(\ref{spher:as}) is indeed important. The reason is that the integrand (in the numerator as well as in the denominator) may have some singularities when ${\rm Re}\,(s_j-s_{j+1})\geq 2$ for all $j=1,\ldots,n-1$ is not satisfied. The absolute integrability is only guaranteed in the stated regime of $s$. For the spherical function $\Psi$ we do not have this restriction since $\det \Pi_{2j,2n}kgg^Tk^T\Pi_{2j,2n}^T$ is positive definite and, hence, invertible for all $j=1,\ldots,n$ because $gg^T$ is positive definite. The analytic continuation of $\Phi$ is possible since $\det \Pi_{2j,2n}kxk^T\Pi_{2j,2n}^T$ is non-negative and $\Phi(s,x)/(\max_{j=1,\ldots,n} a_j)^{\sum_{l=1}^n s_l}$ is bounded in the hyper-half-plane ${\rm Re}\,(s_j-s_{j+1})\geq 2$ for all $j=1,\ldots,n-1$. Here $a\in A$ are the singular values of $x$. The analytic continuation is then given via Carlson's theorem~\cite{Ba35,Mehta}.

The normalization factor in the denominator~\eqref{spher:as} is also important since it is non-trivial for $n>2$, see Eq.~\eqref{a.proof.5} below.  First of all, it simplifies the normalization of the  spherical function which is  $\Phi(s; \imath\eins_n\otimes\tau_2)=1$. For $\Psi$ the normalization is  fixed as $\Psi(s; \eins_{2n})=1$. Moreover this particular normalization allows us to find a very simple and explicit form of the spherical function in terms of the singular values $a\in A$ of $x\in H$. This will be our first main result which is proven in Appendix~\ref{AppA}.

\begin{theorem}[Spherical Function $\Phi$]\label{thm:spher.func}\

Let $a\in A$ and $s\in\mathbb{C}^n$ with non-degenerate spectra, i.e. $a_l\neq a_k$ and $s_l\neq s_k$ for $l\neq k$. The spherical function~\eqref{spher:as} has the explicit form
			\begin{equation}\label{spher:as.b}
				\Phi(s; \imath a\otimes\tau_2)=\left(\prod_{j=0}^{n-1}2^{j}j!\right)\frac{\det[a_c^{s_b+n-1}]_{b,c=1,\ldots,n}}{\Delta_n(a^2)\Delta_n(s)}.
			\end{equation}
\end{theorem}

The explicit result~\eqref{spher:as.b} for the function $\Phi(s; a)$ is the counterpart of the Gelfand-Na\u{\i}mark integral~\cite{GN50} for the spherical function of the unitary group ${\rm U}(n)$. Equation~\eqref{spher:as.b} is helpful in achieving our goal of finding the jPDF of the singular values of the product matrix $y=gxg^T$. There is a deeper group theoretical reason why the  result~\eqref{spher:as.b} can be obtained while the group integral for the spherical function $\Psi(s; g)$ does not permit an analogous
evaluation. The situation is the same as for the relation between the Harish-Chandra~\cite{HC57} and the Itzykson--Zuber~\cite{Itzykson:1980} integrals. While $x$ is in the Lie-algebra of ${\rm O}(2n)$ the matrix $gg^T$ is an element of the coset ${\rm Gl}_{\mathbb{R}}(2n)/{\rm O}(2n)$ which is not a Lie-algebra.

The spherical functions $\Phi$ and $\Psi$ have an important property which is crucial for the identity~\eqref{factor.conv}, and justifies their status as spherical functions, see \cite[Chapter IV]{He00} for the original definition of spherical functions.

\begin{lemma}[Factorization of $\Phi$ and $\Psi$]\label{lem:fact.spher}\

Let $g,g'\in G$, $x\in H$ and $s\in\mathbb{C}^n$. We have
\begin{equation}\label{factorization}
\int_{K}d^*k \, \Phi(s; gkxk^Tg^T)=\Psi(s; g)\Phi(s;x)
\end{equation}
and
\begin{equation}\label{factorization.b}
\int_{K}d^* k \, \Psi(s; gkg'{g'}^Tk^Tg^T)=\Psi(s; g)\Psi(s;g').
\end{equation}
\end{lemma}

\begin{proof}
We first prove Eq.~\eqref{factorization}. For this purpose we consider the case when ${\rm Re}\,z_j\geq 0$ with $z_j=s_j-s_{j+1}-2$ for all $j=1,\ldots,n-1$ to employ the definition~\eqref{spher:as} and then analytically continue in $z_j$. We emphasize that the definition~\eqref{spher:s} of $\Psi$ does not need this restriction since $\det \Pi_{2j,2n}kgg^Tk^T\Pi_{2j,2n}^T$ never vanishes for any $g\in G$, $k\in K$ and $j=1,\ldots,n$.

Let $g\in G$ and $x\in H$. We compute
\begin{equation}
\begin{split}
\int_{K}d^*k&\int_{K}d^*k'\prod_{j=1}^{n} (\det \Pi_{2j,2n}k'gkxk^Tg^T{k'}^T\Pi_{2j,2n}^T)^{z_j/2}\\
=&\int_{K}d^* k\int_{K}d^*k'\prod_{j=1}^{n} (\det \Pi_{2j,2n}RQkxk^TQ^TR^T\Pi_{2j,2n}^T)^{z_j/2}\\
\overset{k\to Q^{-1}k}{=}&\int_{K}d^*k\int_{K}d^*k'\prod_{j=1}^{n} (\det R_{2j}\Pi_{2j,2n}kxk^T\Pi_{2j,2n}^TR_{2j}^T)^{z_j/2}\\
=&\left(\int_{K}d^*k'\prod_{j=1}^{n} (\det R_{2j}R_{2j}^T)^{z_j/2}\right)\left(\int_{K}d^*k\prod_{j=1}^{n} \det\Pi_{2j,2n}kxk^T\Pi_{2j,2n}^T)^{z_j/2}\right),
\end{split}
\end{equation}
where we used the QR-decomposition $k'g=RQ$ with $R$ a lower triangular matrix, $\Pi_{2j,2n}R=R_{2j}\Pi_{2j,2n}$ with $R_{2j}$ the $2j\times2j$ upper left block of $R$, and $Q\in K$ as well as the invariance of the Haar measure. In the last step we redo the QR-decomposition, i.e. $R_{2j}R_{2j}^T=\Pi_{2j,2n}k'gg^T{k'}^T\Pi_{2j,2n}^T$. This is the sought result, except for the analytic continuation.

The analytic continuation can be achieved by dividing $\Phi(s; gkxk^Tg^T)$ by $\max_{j=1,\ldots,n}\{a_j^{\sum_{l=1}^ns_l}\}$ times $\max_{j=1,\ldots,2n}\{\lambda_j^{\sum_{l=1}^ns_l}\}$, where $a\in A$ are the singular values of $x$ and $\lambda\in\mathbb{R}_+^{2n}$ are the singular values of $g$. Dividing $\Phi(s; gkxk^Tg^T)$ by this quantity is equivalent to assuming that the singular values of $g$ and $x$ stay in the interval $[0,1]$ . The integrands of $\Phi$ and $\Psi$ are then bounded, analytic functions in the half planes ${\rm Re}\,z_j\geq 0$ such that we can apply Carlson's theorem~\cite{Ba35,Mehta}.

For Eq.~\eqref{factorization.b}, we apply Eq.~\eqref{factorization} twice to the double group integral
\begin{multline}
\int_{K}d^* k \Psi(s; gkg'{g'}^Tk^Tg^T)\Phi(s;x) \\
\overset{{\rm Eq.}~\eqref{factorization}}{=}\int_{K}d^*k\int_{K}d^*k' \Phi(s; gkg'k'x{k'}^T{g'}^Tk^Tg^T)=\Psi(s; g'{g'}^T)\Psi(s; gg^T)\Phi(s;x).
\end{multline}
Dividing by $\Phi(s;x)$ yields the sought factorization of $\Psi$.
\end{proof}

Both of the spherical functions have a corresponding spherical transform. These spherical transforms are defined similar to the one on ${\rm Gl}_{\mathbb{C}}(N)$, see~\cite{He00,Te88a}.

\begin{definition}[Spherical Transforms corresponding to $\Phi$ and $\Psi$]\label{def:spher.trans}\
 
 Let $P_G\in L^{1,K}(G)$,  $P_H\in L^{1,K}(H)$ and $p_A=\mathcal{I}P_H$.
\begin{enumerate}
\item	 The spherical transform $\mathcal{S}_{\Phi}: L^{1,K}(H)\rightarrow \mathcal{S}_{\Phi}(L^{1,K}(H))$ corresponding to $\Phi$ is defined as
			\begin{equation}\label{Strafo:as}
				\begin{split}
				\mathcal{S}_{\Phi}P_H(s)=&\int_H \frac{dx}{(\det x)^{(2n-1)/2}}P_H(x)\Phi(s; x)\\
				=&\left(\prod_{j=0}^{n-1}2^{j}j!\right)\int_A\frac{da}{(\det a)^{2n-1}} p_A(a)\frac{\det[a_c^{s_b+n-1}]_{b,c=1,\ldots,n}}{\Delta_n(a^2)\Delta_n(s)}=\mathcal{S}_{\Phi}p_A(s)
				\end{split}
			\end{equation}
			for those $s\in\mathbb{C}^n$ for which the integral is defined. In the second line we slightly abuse notation and have to assume that $s_l\neq s_k$ for $l\neq k$.
\item	The spherical transform $\mathcal{S}_{\Psi}: L^{1,K}(G)\rightarrow \mathcal{S}_{\Psi}(L^{1,K}(G))$ corresponding to $\Psi$ is
			\begin{equation}\label{Strafo:s}
			\mathcal{S}_{\Psi}P_G(s)=\int_G \frac{dg}{(\det gg^T)^{(2n-1)/2}}P_G(g)\Psi(s; g)
			\end{equation}
			for those $s\in\mathbb{C}^n$ where the integral exists.	
\end{enumerate} 
\end{definition}

We have chosen the normalizations $\mathcal{S}_\Phi P_H(\{3n-2j\}_{j=1,\ldots,n})=\int_H dx P_H(x)$ and $\mathcal{S}_{\Psi}P_G(\{3n-2j\}_{j=1,\ldots,n})=\int_G dg P_G(g)$. This normalization is similar to that
for the spherical transform on ${\rm Gl}_{\mathbb{C}}(n)$. The particular shift $3n$ off the value $-2j$ reflects the measures used in the integrals~\eqref{Strafo:as} and~\eqref{Strafo:s}. We employ the $G={\rm Gl}_{\mathbb{R}}(2n)$ invariant Haar measure on $H={\rm o}(2n)$ for the spherical transform $\mathcal{S}_{\Phi}$. This choice carries over to the spherical transform $\mathcal{S}_{\Psi}$ allowing for the factorization property \eqref{factor.conv} which will be proven in the next lemma.

\begin{lemma}[Factorization of $S_\Phi$]\label{lem:fact:trans}\

Let $P_G\in L_{\rm Prob}^{1,K}(G)$ and $P_H\in L_{\rm Prob}^{1,K}(H)$. Then the spherical transform of the convolution $P_G\circledast P_H$ is
\begin{equation}\label{factor.conv.b}
\mathcal{S}_{\Phi}[P_G\circledast P_H](s)=\mathcal{S}_{\Psi}P_G(s)\mathcal{S}_{\Phi}P_H(s).
\end{equation}
\end{lemma}

\begin{proof}
The spherical transform of the convolution $P_G\circledast P_H$ is explicitly given as
\begin{equation}
\mathcal{S}_{\Phi}[P_G\circledast P_H](s)=\int_H \frac{dx}{(\det x)^{(2n-1)/2}} \left(\int_G \frac{dg}{(\det gg^T)^{(2n-1)/2}} P_G(g)P_H(g^{-1}x(g^{-1})^T)\right)\Phi(s; x).
\end{equation}
The integrals are absolutely integrable for suitable $s$ allowing the integrals over $g$ and $x$
to be interchanged. Then, we substitute $x\to gxg^T$ and have
\begin{equation}
\mathcal{S}_{\Phi}[P_G\circledast P_H](s)=\int_G \frac{dg}{(\det gg^T)^{(2n-1)/2}}\int_H \frac{dx}{(\det x)^{(2n-1)/2}}  P_G(g)P_H(x)\Phi(s; gxg^T).
\end{equation}
Since $P_H$ is $K$-invariant we can introduce an orthogonal matrix $k\in K$ by $x\to kxk^T$ and integrate over it via the Haar measure on $K$. Hence we have
\begin{equation}
\begin{split}
\mathcal{S}_{\Phi}[P_G\circledast P_H](s)=&\int_G \frac{dg}{(\det gg^T)^{(2n-1)/2}}\int_H \frac{dx}{(\det x)^{(2n-1)/2}}  P_G(g)P_H(x)\int_K d^*k\Phi(s; gkxk^Tg^T)\\
\overset{{\rm Eq.}~\eqref{factorization}}{=}&\int_G \frac{dg}{(\det gg^T)^{(2n-1)/2}}\int_H \frac{dx}{(\det x)^{(2n-1)/2}}  P_G(g)P_H(x)\Psi(s;g)\Phi(s; x).
\end{split}
\end{equation}
The integrals now factorize and yield the claim.
\end{proof}

Another ingredient is needed, namely the inversion of the spherical transform $\mathcal{S}_{\Phi}$ which indeed exists. This is not the case for the spherical transform $\mathcal{S}_{\Psi}$, as can be
seen by considering the $2\times 2$ matrix case ($n=1$). The spherical transforms $\mathcal{S}_{\Phi}$ and $\mathcal{S}_{\Psi}$ are essentially Mellin transforms of the determinants in this case. This is not a problem for $\mathcal{S}_{\Phi}$ because $H$ is a one dimensional space. However $G$ consists of matrices which have two singular values and thus two invariants, the trace and the determinant. The information on the trace is completely lost after the spherical transformation $S_\Psi$. In the following we prove that the inverse of $S_\Phi$ exists.

\newpage
\begin{lemma}[Inverse of $S_\Phi$]\label{lem:inv}\

Let $P_H\in L^{1,K}(H)$, $p_A=\mathcal{I}P_H\in L^{1,\mathbb{S}}(A)$ and define the auxiliary function
\begin{equation}\label{xi.def}
\zeta(z)=\frac{\cos(z)}{\prod_{k=1}^n[1-4z^2/(\pi(2k-1))^2]}.
\end{equation}
The spherical transform $\mathcal{S}_{\Phi}: L^{1,K}(H)\rightarrow \mathcal{S}_{\Phi}(L^{1,K}(H))$ is injective and, hence, invertible and the inverse has the explicit form
\begin{equation}\label{Strafo:as:inv}
\begin{split}
p_A(a)=\mathcal{S}_{\Phi}^{-1}[\mathcal{S}_{\Phi}p_A](a)=&\frac{\Delta_n(a^2)}{(n!)^2 \prod_{j=0}^{n-1}2^{j}j!}\lim_{\epsilon\to0}\int_{\mathbb{R}^n} \frac{ds}{(2\pi)^{n}}\mathcal{S}_{\Phi}P(\{\imath s_j+3n-2j\}_{j=1,\ldots,n})\prod_{l=1}^n\zeta(\epsilon s_l)\\
&\times\Delta_n(\{\imath s_j+3n-2j\}_{j=1,\ldots,n})\det[a_c^{-\imath s_b-2n+2b-1}]_{b,c=1,\ldots,n}
\end{split}
\end{equation}
for almost all $a\in A$.
\end{lemma}

As in~\cite{KK16b}, the regularizing function $\zeta$ can be omitted for the cases when the spherical transform $\mathcal{S}_{\Phi}p_A$ decays sufficiently fast.

\begin{proof}
We only need to show that Eq.~\eqref{Strafo:as:inv} holds for any $p_A\in L^{1,\mathbb{S}}(A)$ and when the singular values $a\in A$ are non-degenerate. The inverse of the spherical transform $\mathcal{S}_{\Phi}$ of $\mathcal{S}_{\Phi}p_A$ is explicitly given as
\begin{multline}
\mathcal{S}_{\Phi}^{-1}[\mathcal{S}_{\Phi}p_A](a) \\
=  \frac{\Delta_n(a^2)}{(n!)^2}\lim_{\epsilon\to0}\int_{\mathbb{R}^n} \frac{ds}{(2\pi)^{n}}\prod_{l=1}^n\zeta(\epsilon s_l)\left[\int_A\frac{d\tilde{a}}{(\det \tilde{a})^{2n-1}} p_A(\tilde{a})\frac{\det[\tilde{a}_c^{\imath s_b+4n-2b-1}]_{b,c=1,\ldots,n}}{\Delta_n(\tilde{a}^2)\Delta_n(\{\imath s_j+3n-2j\}_{j=1,\ldots,n})}\right]\\
\times\Delta_n(\{\imath s_j+3n-2j\}_{j=1,\ldots,n})\det[a_c^{-\imath s_b-2n+2b-1}]_{b,c=1,\ldots,n}.
\end{multline}
Here the integral over $\tilde{a}$ is absolutely integrable because $p_A$ is an $L^1$-function while the
term $\det[\tilde{a}_c^{\imath s_b+2n-2j}]_{b,c=1,\ldots,n}/\Delta_n(\tilde{a}^2)$ is bounded in $\tilde{a}\in A$. The latter can be seen by noticing that its modulus is homogeneous in $\tilde{a}$ of order zero and that the poles of the denominator at the points where $\tilde{a}$ degenerates are compensated by the numerator. Also the integral over $s$ is absolutely integrable due to the regularization $\prod_{l=1}^n\zeta(\epsilon s_l)$. Note that spherical transform $\mathcal{S}_{\Phi}P(\{\imath s_j+3n-2j\}_{j=1,\ldots,n})$ is bounded on $s\in\mathbb{R}^n$. Hence we can interchange the integrals but the limit $\epsilon\to0$ stays in front of both integrals.

We cancel some terms in the numerator with those in the denominator and find
\begin{multline}
\mathcal{S}_{\Phi}^{-1}[\mathcal{S}_{\Phi}p_A](a)= \frac{\Delta_n(a^2)}{(n!)^2\det a}\lim_{\epsilon\to0}\int_A d\tilde{a} \\
\times
\int_{\mathbb{R}^n} \frac{ds}{(2\pi)^{n}}p_A(\tilde{a})\frac{\det[\tilde{a}_c^{\imath s_b+2n-2b}]_{b,c=1,\ldots,n}}{\Delta_n(\tilde{a}^2)}\det[a_c^{-\imath s_b-2n+2b}]_{b,c=1,\ldots,n}\prod_{l=1}^n\zeta_1(\epsilon s_l).
\end{multline}
The integral over $s$ can be performed by Andr\'eief's identity~\cite{An83},
\begin{equation}
\begin{split}
\mathcal{S}_{\Phi}^{-1}[\mathcal{S}_{\Phi}p_A](a)=& \lim_{\epsilon\to0}\frac{\Delta_n(a^2)}{n!\det a}\int_A d\tilde{a} \, p_A(\tilde{a})\frac{\det[(\tilde{a}_c/a_b)^{2n-2b}\mathcal{F}\zeta(\epsilon^{-1}{\rm ln}[\tilde{a}_c/a_b])/\epsilon]_{b,c=1,\ldots,n}}{\Delta_n(\tilde{a}^2)},
\end{split}
\end{equation}
where the inverse Fourier transform of the regularizing function~\eqref{xi.def} is
\begin{equation}
\mathcal{F}\zeta(u)=\frac{1}{2\pi}\int_{-\infty}^\infty dz \, \zeta(z) e^{-\imath z u}=c\Theta(1-u^2)\cos^{2n-1}\left(\frac{\pi u}{2}\right),
\end{equation}
where $\Theta(x)$ denotes the Heaviside step function.
The exact value of the constant $c\neq 0$ is not important. It only correctly normalizes $\mathcal{F}\zeta$ because $\zeta(0)=1$.

Due to the regularization the integration domain shrinks to $A_\epsilon=(\bigcup_{j=1}^n[a_je^{-\epsilon},a_je^{\epsilon}])^n$. We recall that the singular values $a\in A$ are chosen to be non-degenerate. Therefore there is an $\epsilon_0>0$ such that $1/\Delta_n(\tilde{a}^2)$ has no poles for all $\epsilon\leq\epsilon_0$, in particular it is uniformly bounded on $A_{\epsilon_0}$ and, thus, on $A_\epsilon\subset A_{\epsilon_0}$ for any $\epsilon\leq\epsilon_0$. Hence we may expand the determinant in the numerator and get a factor $n!$ due to the symmetry of the integrand, 
\begin{equation}
\begin{split}
\mathcal{S}_{\Phi}^{-1}[\mathcal{S}_{\Phi}p_A](a)=& \lim_{\epsilon\to0}\frac{\Delta_n(a^2)}{\det a}\prod_{l=1}^n\int_{a_le^{-\epsilon}}^{a_le^{\epsilon}} d\tilde{a}_l\ \frac{p_A(\tilde{a})}{\Delta_n(\tilde{a}^2)}\left(\prod_{j=1}^n\left(\frac{\tilde{a}_j}{a_j}\right)^{2n-2j}\frac{c}{\epsilon}\cos^{2n-1}\left(\frac{\pi }{2\epsilon}{\rm ln}\left[\frac{\tilde{a}_j}{a_j}\right]\right)\right).
\end{split}
\end{equation}
Substituting $u_l={\rm ln}(\tilde{a}_l/a_l)^{1/\epsilon}$ then gives
\begin{equation}
\begin{split}
\mathcal{S}_{\Phi}^{-1}[\mathcal{S}_{\Phi}p_A](a)=&\Delta_n(a^2) \lim_{\epsilon\to0}\prod_{l=1}^n\int_{-1}^{1} du_l\ \frac{p_A(\{a_je^{\epsilon u_j}\}_{j=1,\ldots,n})}{\Delta_n(\{a_j^2e^{2\epsilon u_j}\}_{j=1,\ldots,n})}\left(\prod_{j=1}^n c\,\cos^{2n-1}\left(\frac{\pi u_j}{2}\right)e^{\epsilon(2n-2j+1) u_j}\right).
\end{split}
\end{equation}
We now employ the fact that $p_A$ is an $L^1$-function on $A$. The integrand is therefore an $L^1$-function on $[-1,1]^n$. With the same arguments as in the proof of~\cite[Lemma~2.6]{KK16a} we have
\begin{equation}
\begin{split}
\mathcal{S}_{\Phi}^{-1}[\mathcal{S}_{\Phi}p_A](a)=&\Delta_n(a^2) \prod_{l=1}^n\int_{-1}^{1} du_l\ \frac{p_A(\{a_j\}_{j=1,\ldots,n})}{\Delta_n(\{a_j^2\}_{j=1,\ldots,n})}\left(\prod_{j=1}^n c\,\cos^{2n-1}\left(\frac{\pi u_j}{2}\right)\right)=p_A(a)
\end{split}
\end{equation}
for those $a\in A$ which satisfy
\begin{equation}
\lim_{\epsilon\to0}\prod_{l=1}^n\int_{[-1,1]^n} d\tilde{a} |p_A(a+\epsilon\tilde{a})-p_A(a)|=0
\end{equation}
which are almost all. This completes the proof.
\end{proof}

Now we are ready to formulate the theorems about the $G={\rm Gl}_{\mathbb{R}}(2n)$ convolution on $H={\rm o}(2n)$.

\section{General Results for Products involving Polynomial Ensembles on $H$}\label{sec:gen}

We first derive the jPDF's of the singular values for the product matrix $gxg^T$ with $g\in G$ a factorizing ensemble and $x\in H$ either a fixed matrix or drawn from a polynomial ensemble on $H$. This is done in subsection~\ref{sec:jPDF}. We additionally bring it into relation to P\'olya ensembles on ${\rm Gl}_{\mathbb{C}}(n)$, see~\cite{FKK17} for the definition, when $x$ is chosen to be the fixed matrix $\imath \eins_n\otimes\tau_2$.  In subsection~\ref{sec:kernel} we derive the bi-orthogonal systems corresponding to the two kinds of settings. These results are the counterparts to those derived in~\cite{Ki17} for some additive convolutions.

\subsection{Joint Probability Densities}\label{sec:jPDF}

We first calculate the spherical transforms of the two kinds of ensembles in Definition~\ref{def:poly.ens}. For this purpose we recall the univariate Mellin transformation of a function $f\in L^1(\mathbb{R}_+)$,
\begin{equation}\label{Mellin:def}
\mathcal{M}f(s)=\int_0^\infty f(a)a^{s-1} \, da,
\end{equation}
and we introduce the induced density of the modulus of the determinant for a density $h\in L^{1}_{\rm Prob}({\rm Gl}_{\mathbb{R}}(2))$,
\begin{equation}\label{average.def}
\mathcal{A}h(a)=\int_{{\rm Gl}_{\mathbb{R}}(2)}dz \, h(z)\delta(a-\sqrt{\det zz^T}).
\end{equation}
Let us emphasize that $\mathcal{A}h$ is an $L^1$-function on $\mathbb{R}_+$ and so has a well
defined Mellin transformation~\eqref{Mellin:def}. Moreover we note that the sets~\eqref{L1.def.b}, which were used for Definition~\ref{def:poly.ens} of the considered ensembles, are subsets of the $L^1$-functions on $\mathbb{R}_+$ and on ${\rm Gl}_{\mathbb{R}}(2)$, respectively. Thus the transforms are well-defined for the weights associated with these ensembles.

\begin{proposition}[Spherical Transforms of the Ensembles of Definition~\ref{def:poly.ens}]\label{prop:spher.pol}\
\begin{enumerate}
\item	Let $P_H\in L_{\rm Prob}^{1,K}(H)$ be the distribution of a polynomial ensemble on $H$ associated 
with the weights $w_1,\ldots, w_n\in L^1_n(\mathbb{R}_+)$, cf. Eq.~\eqref{poly.def}. The spherical transform of $P_H$ is
			\begin{equation}\label{spher.pol}
			\mathcal{S}_\Phi P_H(s)=n!\left(\prod_{j=0}^{n-1}2^{j}j!\right)C_n[w]\frac{\det[\mathcal{M}w_c(s_b-n+1)]_{b,c=1,\ldots,n}}{\Delta_n(s)}.
			\end{equation}
			The normalization constant is
			\begin{equation}
			\frac{1}{C_n[w]}=n!\det[\mathcal{M}w_c(2b-1)]_{b,c=1,\ldots,n}.
			\end{equation}
\item	Let $P_G \in L_{\rm Prob}^{1,K}(G)$ be the distribution of a polynomial ensemble on $G$ associated with the probability density $\sigma\in L^{1}_n({\rm Gl}_{\mathbb{R}}(2))$. The spherical transform of $P_G$ is given as
			\begin{equation}\label{spher.fac}
			\mathcal{S}_\Psi P_G(s)=\prod_{j=1}^n\frac{\mathcal{M}\circ\mathcal{A}\sigma(s_j-n+1)}{\mathcal{M}\circ\mathcal{A}\sigma(2j-1)}.
			\end{equation}
\end{enumerate}
\end{proposition}

\begin{proof}\
\begin{enumerate}
\item	Let $p_A=\mathcal{I}P_H$. Substitute the jPDF of the singular values $a\in A$ into the second line of Eq.~\eqref{Strafo:as}. The Vandermonde determinant $\Delta_n(a^2)$ cancels and the remaining integrand is absolutely integrable for suitable $s\in\mathbb{C}^n$. We apply Andr\'eief's identity~and find
			\begin{equation}
			\mathcal{S}_\Phi P_H(s)=n!\left(\prod_{j=0}^{n-1}2^{j}j!\right)C_n[w]\frac{\det[\int_0^\infty da_c \, w_c(a_c)a_c^{s_b-n}]_{b,c=1,\ldots,n}}{\Delta_n(s)}.
			\end{equation}
			The integral in the determinant can be identified with the Mellin transformation~\eqref{Mellin:def}. The normalization constant can be obtained by using $\mathcal{S}_\Phi P_H(\{3n-2j\}_{j=1,\ldots,n})=\int_H dx \, P_H(x)=1$.
\item	We start from the formula~\eqref{Strafo:s} and substitute the explicit form~\eqref{spher:s} of the spherical function $\Psi$ into this formula,
			\begin{equation}
			\mathcal{S}_\Psi P_G(s)=\int_G \frac{dg}{(\det gg^T)^{(2n-1)/2}}P_G(g)\int_{K}d^*k\prod_{j=1}^{n} [\det \Pi_{2j,2n}kgg^Tk^T\Pi_{2j,2n}^T]^{(s_j-s_{j+1})/2-1}.
			\end{equation}
			For suitable $s\in\mathbb{C}^n$ both integrals over $g$ and $k$ are absolutely integrable and can be interchanged. Then we can absorb the orthogonal matrix $k$ in $g$ due to the $K$-invariance of $P_G$, i.e.
			\begin{equation}
			\mathcal{S}_\Psi P_G(s)=\int_G dgP_G(g) (\det gg^T)^{(s_n-n)/2}\prod_{j=1}^{n-1} [\det \Pi_{2j,2n}gg^T\Pi_{2j,2n}^T]^{(s_j-s_{j+1})/2-1}.
			\end{equation}
			In the next step we perform the decomposition of $g=t	 z\tilde{k}$ with a lower triangular $2\times 2$ block uni-modular matrix $t\in T$, a real $2\times 2$ block diagonal matrix $z=\diag(z_1,\ldots,z_n)\in Z$ and an orthogonal matrix $\tilde{k}\in K$. Indeed such a decomposition exists since we can first do a QR-decomposition $g=\tilde{t} \tilde{k}$ with $\tilde{t}=tz$ a lower $2\times2$ block diagonal matrix. Then we factor the matrix $\tilde{t}$ by substituting the strictly lower diagonal blocks $\tilde{t}_{ij}$ as $\tilde{t}_{ij}\to \tilde{t}_{ij}z_j$. The Jacobian can also be calculated via these two steps. The QR-decomposition yields the factor $\prod_{j=1}^n(\det z_jz_j^T)^{n-j}$; to apply~\cite[Eq.~(3.6)]{ER05} one needs to apply the QR decomposition for ${\rm Gl}_{\mathbb{R}}(2n)$ and then revert it for the $2\times 2$ blocks $z_j$. The substitution contributes another factor $\prod_{j=1}^n(\det z_jz_j^T)^{n-j}$. Hence, we obtain
			\begin{multline}
			\mathcal{S}_\Psi P_G(s)\propto\int_Z dz\int_T dt\prod_{j=1}^n(\det z_jz_j^T)^{2(n-j)} \\
			\times P_G(tz) (\det tzz^Tt^T)^{(s_n-n)/2}\prod_{j=1}^{n-1} [\det \Pi_{2j,2n}tzz^Tt^T\Pi_{2j,2n}^T]^{(s_j-s_{j+1})/2-1},
			\end{multline}
			where we omitted the normalization constant. Note that the orthogonal matrix $\tilde{k}$ drops out as a consequence of the $K$-invariance. Due to the block triangular form and the uni-modularity of $t$, the determinants can be readily evaluated as
			\begin{equation}
			\det \Pi_{2j,2n}tzz^Tt^T\Pi_{2j,2n}^T=\prod_{l=1}^j \det z_jz_j^T.
			\end{equation}
			Collecting the terms in the integral, we have
			\begin{multline}
			\mathcal{S}_\Psi P_G(s)\propto\int_Z dz\int_T dt\left(\prod_{j=1}^n(\det z_jz_j^T)^{(s_j-n)/2+n-j}\right) P_G(tz) \\
			=\int_Z dz\left(\prod_{j=1}^n(\det z_jz_j^T)^{(s_j-n)/2}\right) \tilde{p}_Z(z),
			\end{multline}
			where we used the definition~\eqref{T.def}, i.e. $\tilde{p}_Z(z)=\mathcal{T}P_G(z)$. For a factorizing ensemble the function $\tilde{p}_Z$ has the form $\tilde{p}_Z(z)=\prod_{j=1}^n\sigma(z_j)$ with $\sigma\in L^{1}_n({\rm Gl}_{\mathbb{R}}(2))$. Hence the integrals factorize into integrals over each single $2\times2$ real matrix $z_j$. The combination of the two integrals~\eqref{Mellin:def} and~\eqref{average.def} and the normalization when choosing $s_j=3n-2j$ yields the claim.
\end{enumerate}
\end{proof}

With the help of the univariate Mellin convolution of two functions $f,h\in L^1(\mathbb{R}_+)$,
\begin{equation}\label{mel.con}
f\circledast h(a)=\int_0^\infty \frac{d\tilde{a}} \, {\tilde{a}} f(\tilde{a})h\left(\frac{a}{\tilde{a}}\right),
\end{equation}
our goal to calculate the jPDF of the singular values of the product $gxg^T$ in our specific setting can be achieved.

\begin{corollary}[JPDF of a Factorizing Ensemble on $G$ times a Polynomial Ensemble on $H$]\label{cor:jPDF.poly}\

Let $g\in G$ be drawn from a factorizing ensemble on $G$ associated with the weight $\sigma\in L^{1}_n({\rm Gl}_{\mathbb{R}}(2))$ and $x\in H$ be drawn from a polynomial ensemble on $H$ associated with the  weights $w_1,\ldots,w_n \in L^1_n(\mathbb{R}_+)$. The product $y=gxg^T\in H$ is a polynomial ensemble on $H$ associated with the weights $\mathcal{A}\sigma\circledast w_1,\ldots,\mathcal{A}\sigma\circledast w_n\in L^1_n(\mathbb{R}_+)$. Specifically the jPDF of the singular values of $gxg^T$ is
\begin{equation}\label{jPDF.fact.poly}
p_A(a)=\frac{C_n[w]}{\prod_{j=1}^n\mathcal{M}\circ\mathcal{A}\sigma(2j-1)}\Delta_n(a^2)\det[\mathcal{A}\sigma\circledast w_b(a_c)]_{b,c=1,\ldots,n}
\end{equation}
for almost all $a\in A$.
\end{corollary}

\begin{proof}
Let $P_G$ and $\tilde{P}_H$ denote the probability densities of $g$ and $x$, respectively. The probability density $P_H$ of $y=gxg^T$ is given by the convolution $P_H=P_G\circledast \tilde{P}_H$. When applying the spherical transform $\mathcal{S}_\Phi$  we have
\begin{equation}
\begin{split}
\mathcal{S}_\Phi[P_H](s)=&\mathcal{S}_\Phi[P_G\circledast \tilde{P}_H](s)\overset{{\rm Lemma}~\ref{lem:fact:trans}}{=}\mathcal{S}_\Psi P_G(s)\mathcal{S}_\Phi\tilde{P}_H(s)\\
=&n!\left(\prod_{j=0}^{n-1}\frac{2^{j}j!}{\sigma(2j+1)}\right)C_n[w]\frac{\det[\mathcal{M}\circ\mathcal{A}\sigma(s_b-n+1)\mathcal{M}w_c(s_b-n+1)]_{b,c=1,\ldots,n}}{\Delta_n(s)}.
\end{split}
\end{equation}
The Mellin transform also satisfies a factorization with respect to the Mellin convolution~\cite{Ti86},
\begin{equation}
\mathcal{M}[f\circledast h]=\mathcal{M}f\,\mathcal{M}h
\end{equation}
for any two $L^1$-functions $f,h\in L^1(\mathbb{R}_+)$. The functions $w_j$ ($j=1,\ldots,n$) and $\mathcal{A}\sigma$ are $L^1$-functions such that we find Eq.~\eqref{jPDF.fact.poly} after applying the inverse $\mathcal{S}_\Phi^{-1}$.
\end{proof}

There is a related result which can be calculated easily with the framework of Sec.~\ref{sec:spher}. Instead of choosing $x$  a random matrix, one can choose it to be a fixed matrix in $H$. 

\begin{theorem}[JPDF of a Factorizing Ensemble on $G$ times a fixed matrix in $H$]\label{thm:jPDF.fixed}\

Let $g\in G$ be drawn from a factorizing ensemble on $G$ associated with the weight $\sigma\in L^{1}_n({\rm Gl}_{\mathbb{R}}(2))$, and $x\in H$ be fixed and invertible with non-degenerate singular values $\tilde{a}\in A$. The jPDF of the singular values of $y=gxg^T\in H$ is
\begin{equation}\label{jPDF.fixed.poly}
p_A(a|\tilde{a})=\frac{1}{n!\prod_{j=1}^{n}\mathcal{M}\circ\mathcal{A}\sigma(2j-1)}\frac{\Delta_n(a^2)}{\Delta_n(\tilde{a}^2)}\det\left[\frac{1}{\tilde{a}_c}\mathcal{A}\sigma\left(\frac{a_b}{\tilde{a}_c}\right)\right]_{b,c=1,\ldots,n}
\end{equation}
for almost all $a\in A$ and, hence, $y$ is drawn from a polynomial ensemble on $H$.
\end{theorem}

As usual one can readily obtain the case with a degenerate $x$ by applying l'H\^opital's rule. The limit to non-invertible $x$ is more subtle since one or more singular values of $y$ become zero as well; see the discussion in Sec.~\ref{sec:conclusio}.

\begin{proof}
By the $K$ invariance of $G$ we assume that $x$ is already of the  quasi-diagonal form $\imath\tilde{a}\otimes\tau_2$. We consider the spherical transform of $y=g(\imath\tilde{a}\otimes\tau_2)g^T$ which is
\begin{equation}
\begin{split}
\mathcal{S}_\Phi[p_A(a|\tilde{a})](s)=&\int_H\frac{dy}{(\det y)^{(2n-1)/2}}\left(\int_G dg P_G(g)\delta(y-g(\imath\tilde{a}\otimes\tau_2)g^T)\right)\Phi(s;y)\\
=&\int_G \frac{dg}{(\det g(\imath\tilde{a}\otimes\tau_2)g^T)^{(2n-1)/2}} P_G(g)\Phi(s;g(\imath\tilde{a}\otimes\tau_2)g^T)
\end{split}
\end{equation}
for suitable $s\in\mathbb{C}^n$ and $\delta$ the Dirac delta function on $H$.
Since $P_G$ is $K$-invariant we can introduce an orthogonal matrix $k\in K$  via $g\to gk$ and integrate over it with the Haar measure on $K$. The integration over $k$ can be interchanged with the one over $g$ because they are absolutely integrable for suitable $s\in\mathbb{C}^n$. Due to Lemma~\ref{lem:fact.spher}, in particular Eq.~\eqref{factorization}, we have
\begin{equation}
\begin{split}
\mathcal{S}_\Phi[p_A(a|\tilde{a})](s)=&\int_G \frac{dg}{(\det gg^T)^{(2n-1)/2}(\det \tilde{a})^{2n-1}} P_G(g)\Psi(s;gg^T)\Phi(s,\imath\tilde{a}\otimes\tau_2)\\
\overset{{\rm Eqs.}~\eqref{spher:as.b},~\eqref{spher.fac}}{=}&\left(\prod_{j=0}^{n-1}\frac{2^j j!}{\mathcal{M}\circ\mathcal{A}\sigma(2j+1)}\right)\frac{\det[\mathcal{M}\circ\mathcal{A}\sigma(s_b-n+1)\tilde{a}_c^{s_b-n}]_{b,c=1,\ldots,n}}{\Delta_n(\tilde{a}^2)}.
\end{split}
\end{equation}
This result can be compared with the spherical transform of a polynomial ensemble on $H$, see Eq.~\eqref{spher.pol}. Thus we have only to perform the inverse Mellin transform of the functions in the determinant. For this purpose we can use
\begin{equation}
\int_0^\infty \frac{da}{a}  \, a^{s_b-n+1} \mathcal{A}\sigma\left(\frac{a}{\tilde{a}_c}\right)=\tilde{a}_c^{s_b-n+1}\mathcal{M}\circ\mathcal{A}\sigma(s_b-n+1).
\end{equation}
Since two $L^1$-functions which agree with their Mellin transform also agree almost everywhere we have completed the proof.
\end{proof}

A particular limit of Theorem~\ref{thm:jPDF.fixed} is of special interest, namely when $a\to\eins_n$. Then we obtain a jPDF which resembles those for P\'olya ensembles on  ${\rm Gl}_{\mathbb{C}}(n)$, see~\cite{KR16,FKK17,Ki17}. Indeed this result shows that the weight $\mathcal{A}\sigma(e^y)$ with $y\in\mathbb{R}$ has to be P\'olya function~\cite{Polya:1913,Polya:1915,Schoenberg:1951}, for the same reasons as discussed in~\cite{KK16b}.

\begin{corollary}[Limit $a\to\eins_n$ of Theorem~\ref{thm:jPDF.fixed}]\label{corr.lim}\

We consider the setting of Theorem~\ref{thm:jPDF.fixed} apart from  $x=\imath \eins_n\otimes\tau_2\in H$. Moreover we assume that $\mathcal{A}\sigma$ is $(n-1)$-times differentiable. Then, the jPDF of the singular values of $y=gxg^T\in H$ is
\begin{equation}\label{jPDF.deg.poly}
p_A(a)=\frac{1}{2^{n(n-1)/2}n!\prod_{j=1}^{n}\mathcal{M}\circ\mathcal{A}\sigma(2j-1)}\Delta_n(a^2)\det\left[(-a_b\partial_{a_b})^{c-1}\mathcal{A}\sigma\left(a_b\right)\right]_{b,c=1,\ldots,n}.
\end{equation}
\end{corollary}

\begin{proof}
After applying l'H\^opital's rule we employ the identity
\begin{equation}
\left.\partial_t^{c-1}\frac{1}{t}\mathcal{A}\sigma\left(\frac{a_b}{t}\right)\right|_{t=1}=\prod_{j=1}^{c-1}(-a_b\partial_{a_b}-j)\mathcal{A}\sigma\left(a_b\right).
\end{equation}
This identity can be proven first for monomials and, then, can be extended to arbitrary differentiable functions. The product is a polynomial of order $c-1$ in the differential operator $-a_b\partial_{a_b}$ such that we can also employ the monomials as a basis in the second determinant of the jPDF yielding Eq.~\eqref{jPDF.deg.poly}.
\end{proof}

Let us point out that the jPDF~\eqref{jPDF.deg.poly} is equal to the jPDF of the singular values (not the squared singular values!) of a P\'olya ensemble on ${\rm Gl}_{\mathbb{C}}(n)$ associated to weight function $x\mathcal{A}\sigma\left(x^2\right)$ with $x\in\mathbb{R}_+$, cf.~\cite[Definition 3.6]{FKK17}. We do believe that this is not a coincidence though we have not found a direct mapping. However we want to point out a group theoretical argument that there has to be a relation. The group invariance of the matrix $g$ in the product $g(\imath \eins_n\otimes\tau_2)g^T$ under right multiplication with $k\in K$ yields a subgroup of the orthogonal group $K={\rm O}(2n)$ which satisfies the invariance equation $k(\imath \eins_n\otimes\tau_2)k^T=(\imath \eins_n\otimes\tau_2)$. This subgroup is the unitary group ${\rm U}(n)$ in its real representation, meaning for a $U\in {\rm U}(n)$ we identify
\begin{equation}
k=\left[\begin{array}{cc} {\rm Re}\,U & {\rm Im}\, U \\ -{\rm Im}\, U & {\rm Re}\,U \end{array}\right]\in K.
\end{equation}
This identification directly follows from the orthogonality condition of $k$ and the invariance condition. It might be that this is the crucial reason for direct identification of the two jPDFs.

\subsection{Bi-Orthonormal Functions and Kernels}\label{sec:kernel}

Here we pursue the ideas of~\cite{Ki17}.  According to standard terminology, the
family of pairs of functions $\{p_j,q_j\}_{j=0,\ldots,n-1}$ is termed bi-orthonormal (on $\mathbb{R}_+$) when we have
\begin{equation}
\int_0^\infty p_l(a)q_k(a) \,da=\delta_{lk},\ {\rm for\ all}\ k,l=0,\ldots,n-1.
\end{equation}
A polynomial ensemble on $H$ associated with the weights $\{w_b\}_{b=1,\ldots,n}$ is described by $\{p_j,q_j\}_{j=0,\ldots,n-1}$ when $p_0,\ldots,p_{n-1}$ builds a basis of the even polynomials up to order $2n-2$ and the functions $q_0,\ldots,q_{n-1}$ span the same space of functions as the weights $w_1,\ldots,w_n$. The $k$-point correlation function is then simply given by~\cite{Bor99}
\begin{equation}\label{kpoint}
R_k(a_1,\ldots,a_k)=\det[K_n(a_l,a_k)]_{l,k=1,\ldots,k}
\end{equation}
with the kernel
\begin{equation}\label{kernel}
K_n(a_l,a_k)=\sum_{j=0}^{n-1} p_j(a_l)q_{j}(a_k).
\end{equation}

We first give one particular pair of bi-orthonormal functions for the jPDF~\eqref{jPDF.fact.poly}. It looks very similar to the one in~\cite[Corollary 3.7]{KK16b} which was derived for the product of a P\'olya ensemble (polynomial ensemble of derivative type) with a polynomial ensemble on ${\rm Gl}_{\mathbb{C}}(n)$. For this purpose we define the polynomial
\begin{equation}\label{Laurent}
\chi_{m_1,m_2}([\sigma],z)=\sum_{j=m_1}^{m_2}\frac{z^{2j}}{\mathcal{M}\circ\mathcal{A}\sigma(2j+1)}.
\end{equation}
In the case that the Laurent series exists we can set $m_1=-\infty$ and $m_2=\infty$ otherwise we choose it to be $m_1=0$ and $m_2=n-1$. Here we set $1/\mathcal{M}\circ\mathcal{A}\sigma(s)=0$ when $\mathcal{M}\circ\mathcal{A}\sigma(s)=\infty$.

\begin{theorem}[Eigenvalue Statistics with Polynomial Ensemble on $H$]\label{thm:biorth.poly}\

Consider the jPDF~\eqref{jPDF.fact.poly}. Let the polynomial ensemble 
specified by $x$
be described by the bi-orthonormal functions $\{\tilde{p}_j,w_{j+1}\}_{j=0,\ldots,n-1}$ and let 
the corresponding kernel be denoted by $\tilde{K}_n$. The polynomial ensemble specified by Eq.~\eqref{jPDF.fact.poly} is described by the pair of bi-orthogonal functions
\begin{equation}\label{biorth.poly}
\{p_j,q_j\}_{j=0,\ldots,n-1}=\biggl\{\oint\frac{dz}{2\pi\imath z}\chi_{0,n-1}([\sigma],z)\tilde{p}_j\left(\frac{y'}{z}\right), \mathcal{A}\sigma\circledast w_{j+1} \biggl\}_{j=0,\ldots,n-1}.
\end{equation}
The kernel has the double contour integral form
\begin{equation}\label{kernel.poly}
K_n(y',y)=\oint\frac{dz}{2\pi\imath z}\int_0^\infty\frac{da}{a} \chi_{0,n-1}([\sigma],z)\mathcal{A}\sigma(a)\tilde{K}_n\left(\frac{y'}{z},\frac{y}{a}\right).
\end{equation}
The contour for the $z$-integration only encircles the origin counter-clockwise.
\end{theorem}

\begin{proof}
To check the bi-orthonormality we need to consider the integral
\begin{equation}
\begin{split}
 \int_0^\infty dy \, p_k(y)q_l(y)=&\int_0^\infty dy p_k(y)\left(\int_0^\infty \frac{da}{a} \mathcal{A}\sigma(a) w_{l+1}\left(\frac{y}{a}\right)\right)\\
 =&\int_0^\infty dy w_{l+1}(y)\left(\int_0^\infty da\mathcal{A}\sigma(a) p_{k}\left(ay\right)\right).
\end{split}
\end{equation}
We employed the fact that both integrals are absolutely integrable and, thus, can be interchanged, in particular we could rescale $y\to ay$. In the next step we explicitly write the polynomials $p_k$ and $\tilde{p}_k$ as sums, i.e.
\begin{equation}
\tilde{p}_k(y)=\sum_{j=0}^k d_{kj}y^{2j}\ \Rightarrow\ p_k(y)=\sum_{j=0}^k \frac{d_{kj}}{\mathcal{M}\circ\mathcal{A}\sigma(2j+1)}y^{2j}.
\end{equation}
The integral over $a$ for each single monomial yields $\mathcal{M}\circ\mathcal{A}\sigma(2j+1)$ which cancels such that we are left with
\begin{equation}
\int_0^\infty dy \, p_k(y)q_l(y)=\int_0^\infty dy \, \tilde{p}_k(y)w_{l+1}(y)=\delta_{kl},
\end{equation}
which is the bi-orthonormality. The kernel simply follows by substituting the definition of the Mellin convolution~\eqref{Mellin:def} and the contour integral~\eqref{biorth.poly} for the polynomials into Eq.~\eqref{kernel}. The integrals can be interchanged with the sum since the convergence is uniform.
\end{proof}

The next case we want to consider is the product which involves a fixed matrix $x\in H$, in particular we want to give a pair of bi-orthogonal functions which describes the jPDF~\eqref{jPDF.fixed.poly}.

\begin{theorem}[Eigenvalue Statistics with a Fixed Matrix in $H$]\label{thm:biorth.fixed}\

Consider the jPDF~\eqref{jPDF.fixed.poly} with $a\in A$ the distinguished eigenvalues of $x$. This polynomial ensemble is specified by the pair of bi-orthogonal functions
\begin{equation}\label{biorth.fixed}
\{p_j,q_j\}_{j=0,\ldots,n-1}=\biggl\{\oint\frac{dz}{2\pi\imath z}\chi_{0,n-1}([\sigma],z)\prod_{i\neq j}\frac{a_i^2-(y'/z)^2}{a_i^2-a_j^2}, \frac{1}{a_j}\mathcal{A}\sigma\left(\frac{y}{a_j}\right) \biggl\}_{j=1,\ldots,n},
\end{equation}
and the kernel is
\begin{equation}\label{kernel.fixed}
K_n(y',y)=\oint\frac{dz}{2\pi\imath z}\chi_{0,n-1}([\sigma],z)\left[\sum_{j=1}^n \frac{1}{a_j}\mathcal{A}\sigma\left(\frac{y}{a_j}\right) \prod_{i\neq j}\frac{a_i^2-(y'/z)^2}{a_i^2-a_j^2}\right].
\end{equation}
As before, the contour for the $z$-integration only encircles the origin counter clockwise.
\end{theorem}

Let us emphasize two things. First, we can again replace $\chi_{0,n-1}([\sigma],z)$ by $\chi_{-\infty,\infty}([\sigma],z)$ when the Laurent series~\eqref{Laurent} exists for a suitable radius. Second, the polynomials $p_j$ are all of order $2n-2$ as was found for similar random matrices involving fixed matrices, see~\cite{Akemann:2015b,Ki17}.

\begin{proof}
As in the proof of Theorem~\ref{thm:biorth.poly} we consider the integral
\begin{equation}
\begin{split}
 \int_0^\infty dy \,p_k(y)q_l(y)=&\int_0^\infty dy \left(\oint\frac{dz}{2\pi\imath z}\chi_{0,n-1}([\sigma],z)\prod_{i\neq k}\frac{a_i^2-(y/z)^2}{a_i^2-a_k^2}\right)\frac{1}{a_l}\mathcal{A}\sigma\left(\frac{y}{a_l}\right)\\
 =&\int_0^\infty dy \left(\oint\frac{dz}{2\pi\imath z}\chi_{0,n-1}([\sigma],z)\prod_{i\neq k}\frac{a_i^2-(a_ly/z)^2}{a_i^2-a_k^2}\right)\mathcal{A}\sigma\left(y\right),
\end{split}
\end{equation}
where we have rescaled $y\to a_l y$. We can now expand the product in $y/z$ and integrate first over $z$ and then over $y$. The contour integral over $z$ yields a factor $1/\mathcal{M}\circ\mathcal{A}\sigma(2j+1)$ for the monomial $(y/k)^{2j}$ while the integral over $y$ cancels this term. Hence we have
\begin{equation}
\begin{split}
 \int_0^\infty dy \, p_k(y)q_l(y)=&\prod_{i\neq k}\frac{a_i^2-a_l^2}{a_i^2-a_k^2}.
\end{split}
\end{equation}
This product vanishes always when $l\neq k$ while it is unity for $l=k$. This proof is finished by substituting the contour integral into Eq.~\eqref{kernel} and interchanging it with the sum.
\end{proof}

We can simplify the kernel~\eqref{kernel.fixed} when we assume an analyticity property of the weight $\mathcal{A}\sigma$.  Although this is not always satisfied,  it is shared by some prominent cases like the Ginibre ensemble and the Jacobi ensemble in the open interval $]0,1[$.

\begin{corollary}[Simplification of the Kernel~\eqref{kernel.fixed}]\label{cor:ker.fixed}\

We assume the setting of Theorem~\ref{thm:biorth.fixed}. Additionally let $\mathcal{A}\sigma$ be holomorphic in an open set around the points $y/a_1,\ldots,y/a_n$ with a fixed $y\in\mathbb{R}_+$. Then the kernel~\eqref{kernel.fixed} simplifies to
\begin{equation}\label{kernel.fixed.b}
K_n(y',y)=\oint\frac{dz'}{2\pi\imath z'}\oint\frac{dz}{\pi\imath }\chi_{0,n-1}\left([\sigma],\frac{y'}{z'}\right)\mathcal{A}\sigma\left(\frac{y}{z}\right)\frac{1}{{z'}^2-z^2} \prod_{i=1}^n\frac{a_i^2-{z'}^2}{a_i^2-z^2},
\end{equation}
where $z'$ only encircles the origin counter clockwise and $z$ encircles only the poles at $a_1,\ldots, a_n$ (but none of the others) counter clockwise.
\end{corollary}

\begin{proof}
 The proof immediately follows from
 \begin{equation}
 \frac{1}{a_j^2-z^2}=\frac{1}{2a_j}\left[\frac{1}{z+a_j}-\frac{1}{z-a_j}\right].
 \end{equation}
 This can be done for each of the poles $z=a_1,\ldots,a_n$. Moreover we change the variable $z'\to y'/z'$.
\end{proof}

We remark that the structure of Eq.~\eqref{kernel.fixed.b} is analogous to results in
the recent work \cite{Ki17} where sums of Hermitian random matrices corresponding to all three classical Lie algebras and sums of complex random matrices with
fixed matrices were considered. Furthermore we want to underline that the bi-orthonormal functions for the fully degenerate case $a\to\eins_n$ follows from~\cite[Lemma 4.2]{KK16a} because the jPDF~\eqref{jPDF.deg.poly} has the same form as the one of a P\'olya ensemble (multiplicative type) on ${\rm Gl}_{\mathbb{C}}(n)$ when substituting $a\to\sqrt{a}$.

\section{Examples}\label{sec:examples}

In this section we want to illustrate the general theory of the previous sections with two classical ensembles. The first is the Gaussian case (induced Ginibre ensemble), see subsection~\ref{sec:Gin} and the second is the induced real Jacobi ensemble (truncated orthogonal matrices), see subsection~\ref{sec:Jac}. In subsection~\ref{sec:prod} we consider products of truncated orthogonal matrices.

\subsection{Products with Induced Real Ginibre Matrices}\label{sec:Gin}

We consider a rectangular real Ginibre matrix $M$ say of even row and column sizes is a $2N \times 2n$ $(N \ge n)$
matrix with independent, standard Gaussian entries. The corresponding induced ensemble is defined
as the set of random matrices of the form $g = R (M^T M)^{1/2}$, where $R \in K={\rm O}(2n)$.
We know from \cite{FBKSZ12} that the probability density of $g$ is proportional to
$ (\det gg^T)^\nu e^{- {\rm Tr} \, gg^T/2}$, $\nu = N - n$.
The rectangular Ginibre ensemble as well as the results below can be analytically continued to any $\nu>-1/2$.

It is well-known that the induced real Ginibre ensemble specifies a factorizing ensemble such that the associated weight $\sigma$ is of the form~\cite{Ed97,Fi12}
\begin{equation}\label{Gin.sig}
\sigma(z)\propto (\det zz^T)^\nu e^{- {\rm Tr} \, zz^T/2}.
\end{equation}
For $x \in H$ fixed, application of Theorem \ref{thm:jPDF.fixed} gives the explicit form of the 
jPDF of the singular values for the random
product matrix $g x g^T$. 

\begin{corollary}[JPDF with an Induced Ginibre Matrix]\label{corr3.1}\

Let $x \in H$ be fixed, and suppose it is invertible with distinct singular values $\tilde{a}\in A$.
For $g \in G$
drawn from the induced Ginibre ensemble we
have that the singular values of $gxg^T$, say $a_j\in A$, have the jPDF
\begin{equation}\label{3.2}
p_A(a|\tilde{a})=\frac{1}{n!}
\prod_{l=0}^{n-1} {1 \over{(2\nu +2l)!}}
{\Delta_n (a^2) \over \Delta_n(\tilde{a}^2)}\frac{\det a^{2\nu}}{\det\tilde{a}^{2\nu+1}}
\det \Big [ e^{- a_j/ \tilde{a}_k} \Big ]_{j,k=1,\ldots,n}.
\end{equation}
\end{corollary}

\begin{proof}
According to Eq.~(\ref{jPDF.fixed.poly}), we must compute Eq.~(\ref{average.def}) with $h(z)$ specified by the
right hand side of Eq.~(\ref{Gin.sig}),
\begin{equation}\label{3.1}
{\mathcal A}\sigma(a) \propto \int_{{\rm Gl}_{\mathbb R}(2)} dz \,
\delta(a - \sqrt{ \det z z^T}) (\det z z^T)^\nu e^{- {\rm Tr} \, z z^T/2}.
\end{equation}
The fact that the $a$ independent terms of ${\mathcal A}\sigma(a) $ cancel from
Eq.~(\ref{jPDF.fixed.poly}) allows such terms to be ignored, as indicated in
 Eq.~(\ref{3.1}) by use of
the proportionality symbol ``$\propto$".

To evaluate the integral we make use of the $2 \times 2$ QR-decomposition, by writing
$z = k \begin{bmatrix} t_{11} & t_{12} \\ 0 & t_{22} \end{bmatrix}$ where $k \in {\rm O}(2)$,
$t_{11}, t_{22} \ge 0$ and $t_{12} \in \mathbb R$. This change of variables gives
$dz = t_{11} dt_{11} dt_{22} dt_{12} d^*k$ for the transformation of the measure;
see e.g.~\cite{ER05}. Also, the integrand is independent of $k$, while the
dependence on $t_{12}$ factorises. Integrating over these variables gives
\begin{align}\label{3.4}
{\mathcal A}\sigma(a) & \propto  \int_0^\infty dt_{11} \int_0^\infty dt_{22} \,
t_{11} \delta(a - t_{11} t_{22}) (t_{11} t_{22})^{2\nu} e^{-(t_{11}^2 + t_{22}^2 )/2} \nonumber \\
& \propto   a^{2\nu} \int_0^\infty dt_{11}  \,  e^{-(t_{11}^2 +a^2/ t_{11}^2)/2} \nonumber \\
& \propto a^{2 \nu}  e^{- a},
\end{align}
where the final line follows from a special case of an integral evaluation due to Boole;
see e.g.~\cite[Eq.~(11.49)]{Fo10}. Substituting this in Eq.~(\ref{Mellin:def}) shows
\begin{equation}\label{3.5}
\mathcal M \circ \mathcal A \sigma(2j+1) = \int_0^\infty e^{-a} a^{{2 \nu + 2j}} \, da =
{(2\nu + 2j)!}.
\end{equation}
Now substituting both Eqs.~(\ref{3.4}) and (\ref{3.5}) in Eq.~(\ref{jPDF.fixed.poly}), Eq.~(\ref{3.2}) results.

\end{proof}

We remark that the result of Corollary \ref{corr3.1} includes the result of \cite[Corollary 4.3]{FILZ17},
which (after a minor change of notation) tells us that with $X$ a $2N \times 2n$ ($N \ge n$) standard
real Gaussian matrix, and 
\begin{align}
A=\operatorname{diag}(a_1,\ldots,a_n)\otimes
\begin{bmatrix}
0 & \imath \\
-\imath & 0
\end{bmatrix}
\end{align}
with each $a_j > 0$ and distinct, the jPDF of the eigenvalues of $XAX^T$ is given by
Eq.~(\ref{3.2}) with $\nu = N - n$. 
In the limiting case corresponding to $a=\eins_n$ (see Proposition \ref{jPDF.deg.poly} for this limit)
one obtains a particular Laguerre
Muttalib-Borodin ensemble, as first derived in \cite{LSZ06} and
\cite{De10} using different methods. In \cite{FILZ17} this latter result was used to derive the eigenvalue PDF
of the random product matrices
\begin{equation}\label{mn1}
X_M \cdots X_1  \Big ( \eins_n \otimes \tau_2 \Big ) 
X_1^T \cdots X_M^T \quad {\rm and} \quad
X_M \cdots X_1  (\imath \tilde{A})
X_1^T \cdots X_M^T,
\end{equation}
where for
$j=1,\dots,M$, $X_j$ denotes a real standard Gaussian matrix of size
$2 (n + \nu_j) \times 2 (n + \nu_{j-1})$ with $\nu_j \ge \nu_{j-1}$ with $\nu_0 = 0$,
and $\tilde{A}$ is a $2n \times 2n$ real standard anti-symmetric matrix.

In fact, up to rescaling, the square of the eigenvalues for these matrices were shown to be
identical in distribution to the eigenvalues of the random product matrices
\begin{equation}\label{mn2}
G_{2M}^\dagger\cdots G_1^\dagger G_1\cdots G_{2M} \quad {\rm and} \quad
G_{2M+1}^\dagger\cdots G_1^\dagger G_1\cdots G_{2M+1}
\end{equation}
respectively, where the $G_i$ are particular complex Gaussian matrices. The latter have
jPDF proportional to
\begin{equation}\label{4.39}
\det (G_i^\dagger G_i)^{\nu_i-1/2} e^{-\Tr G_i^\dagger G_i}
\qquad \text{and} \qquad
\det (G_i^\dagger G_i)^{\nu_i} e^{-\Tr G_i^\dagger G_i}
\end{equation}
for $i=2j-1$ ($j=1,\ldots,m+1$) with $\nu_{2M+1}=0$ and $i=2j$ ($j=1,\ldots,m$).
As a result, the corresponding bi-orthogonal system is known from \cite{AIK13} in terms of
Meijer G-functions, as is a double contour form of the kernel \cite{KZ14}. These are consistent
with the forms implied by Theorem \ref{thm:biorth.poly}.

\subsection{The Case of  Induced Real Jacobi Matrices}\label{sec:Jac}

Consider a $K_1 \times K_1$ real orthogonal matrix chosen with Haar measure.
Delete $\ell_N = K_1 - 2N$ rows and $\ell_n = K_1 - 2n$ columns, leaving a $2N \times 2n$
rectangular matrix $M$, and suppose for definiteness that
$N \ge n$. The jPDF for the singular values specifies the real Jacobi ensemble.
Analogous to the Ginibre case, the corresponding induced ensemble is defined
as the set of random matrices of the form $g = R (M^T M)^{1/2}$, where $R \in K={\rm O}(2n)$.
These matrices are distributed \cite{KSZ09,Fi12} according to the jPDF proportional to
$(\det gg^T)^\nu \det(\eins_{2n}-gg^T)^\mu\Theta(\eins_{2n}-gg^T)$, where 
\begin{equation}\label{mn}
\nu = (N-n), \qquad
\mu = (K_1 - 2n - 2N-1)/2,
\end{equation}
which thus requires $K_1 \ge 2(n + N)$ to be well defined (for
$K_1 < 2(n + N)$ some of the singular values of $g$ will equal unity with probability one).
Nonetheless, the jPDF~(\ref{jPDF.Z}) is well defined independent of this condition,
and furthermore analytic continuation off the integers gives meaning to  the parameter range  $\nu>-1/2$ and $\mu>-n-1/2$.

As for the induced real Ginibre ensemble, the induced real Jacobi ensemble specifies a factorizing ensemble such that the associated weight $\sigma$ is of the same general
form as the jPDF for $g$ \cite{KSZ09,Fi12},
\begin{equation}\label{Gin.sig1}
\sigma(z)\propto (\det zz^T)^\nu \det(\eins_{2}-zz^T)^{\mu+n-1}\Theta(\eins_{2}-zz^T).
\end{equation}
In further analogy with the induced real Ginibre matrices, the quantity $\mathcal A \sigma$ in
 Theorem \ref{thm:jPDF.fixed} can be made explicit, thus allowing the jPDF of the singular values of
 the random matrix $g x g^T$, with $x \in H$, to be specified. 
 
 \begin{corollary}[JPDF with an Induced Jacobi Matrix]\label{corr3.2}\
 
Let $x \in H$ be fixed, and suppose it is invertible with distinct singular values $\tilde{a}\in A$.
For $g \in G$
drawn from the induced Jacobi ensemble as above described, we
have that the  singular values $\{a_j\}_{j=1}^n$ of
$gxg^T$ have the jPDF
\begin{equation}\label{3.2a}
p_A(a|\tilde{a})=\frac{1}{n!} C_n(\nu,\mu)
{\Delta_n (a^2) \over \Delta_n(\tilde{a}^2)}
\frac{\det a^{2 \nu}}{\det\tilde a^{2\nu+1}}
\det \Big [ (1 - a_j/ \tilde{a}_k)^{2 (\mu  + n ) } 
\Theta(\tilde{a}_k-a_j) \Big ]_{j,k=1,\ldots,n},
\end{equation}
where $\nu, \mu$ are given by Eq.~(\ref{mn}) and
\begin{equation}\label{3.2aa}
C_n(\nu,\mu) = \prod_{j=1}^{n} 
{\Gamma(2\nu + 2\mu + 2n +2 j) \over \Gamma(2 \nu + 2j-1) \Gamma(2\mu + 2n +1 )}
\end{equation}
is the normalization constant.
\end{corollary}

\begin{proof}
Neglecting $a$ independent factors as done in Eq.~(\ref{3.1}), the main task is to evaluate
\begin{equation}\label{3.1a}
{\mathcal A}\sigma(a) \propto \int_{{\rm Gl}_{\mathbb R}(2)} dz \,
\delta(a - \sqrt{ \det z z^T})  (\det zz^T)^\nu \det(\eins_{2}-zz^T)^{\mu+n-1}\Theta(\eins_{2}-zz^T) .
\end{equation}
Again, 
we make use of the $2 \times 2$ QR-decomposition. After minor simplifications, this shows
\begin{equation}
\begin{split}
 {\mathcal A}\sigma(a) \propto &\int_0^\infty dt_{11} \int_0^\infty dt_{22}  \int_{-\infty}^\infty dt_{12} \,
t_{11} \delta ( a -t_{11} t_{22} ) (t_{11} t_{22})^{2\nu} \\
&\times ((1 - t_{11}^2)(1 - t_{22}^2) - t_{12}^2)^{\mu+n-1}\Theta(1-t_{11} )
\Theta((1 - t_{11}^2)(1 - t_{22}^2) - t_{12}^2).
\end{split}
\end{equation}
Now we change variables $t_{12} = (1 - t_{11}^2)^{1/2}(1 - t_{22}^2)^{1/2} s$. This gives a factorization of the variable $s$,
contributing only to the proportionality after integration. The integral over $t_{22}$ is then
straightforward, leaving us with
\begin{equation}
 {\mathcal A}\sigma(a) \propto  a^{2\nu}  \int_a^1dt_{11} \,
  ((1 - t_{11}^2)(1 - (a/t_{11})^2))^{\mu+n-1/2}, \qquad 0 < x < 1.
\end{equation}
  For integer values of $\mu + n-1/2$ inspection of the integrand shows that this is a polynomial
  in $a$ of degree $2(\mu + n)$ which furthermore vanishes at $a=1$ with exponent equal to this
  same degree. Hence, we must have
\begin{equation}\label{3.4a}  
 {\mathcal A}\sigma(a) \propto  a^{2\nu} (1 - a)^{2(\mu + n) } \Theta(a(1-a));
\end{equation}
 use of Carlson's theorem guarantees that this evaluation remains true
  for general $\mu$ such that the integral is well
 defined.

 Substituting Eq.~(\ref{3.4a}) in Eq. (\ref{Mellin:def}) shows
\begin{equation}\label{3.5a}
\mathcal M \circ \mathcal A \sigma(2j+1) = \int_0^1  (1 - x)^{2(\mu + n) } x^{2 \nu + 2j} \, dx =
{\Gamma(2 \nu + 2j+1) \Gamma(2\mu + 2n +1 ) \over \Gamma(2\nu + 2\mu + 2n +2 j+2)}.
\end{equation}
Again we substitute both Eqs.~(\ref{3.4}) and (\ref{3.5}) in Eq.~(\ref{jPDF.fixed.poly}) and find Eq.~(\ref{3.2a}) with the constant~\eqref{3.2aa}.
 \end{proof}
 
 The simplest example of the class of random matrices in Corollary \ref{corr3.2} occurs when
 \begin{equation}\label{3.6a}
\imath x =   \eins_n \otimes \tau_2.
\end{equation} 
The corresponding jPDF of the singular values follows by taking the limit $\tilde{a}_k \to 1$ for each
$k=1,\dots, n$, making use of l'H\^opital's rule. Alternatively, one can appeal to Corollary~\ref{corr.lim}.

 \begin{corollary}\label{corr3.3}
 Consider the setting of Corollary \ref{corr3.2}, and choose $\imath x$ according to Eq.~(\ref{3.6a}).
 The singular values of $gxg^T$ have the jPDF
\begin{multline}\label{3.1c}
p_A(a) = \frac{1}{2^{n(n-1)/2}n!}
\prod_{j=1}^{n} 
{\Gamma(2\nu + 2\mu + 2n +2 j) \over \Gamma(j)\Gamma(2 \nu + 2j-1) \Gamma(2\mu + 2n +2-j )} \\
\times
\Delta_n (a^2) \Delta_n(a) \det a^{2 \nu} \det(\eins_n - a)^{2 \mu + n + 1} \Theta(\eins_n-a).
\end{multline}
\end{corollary}

This is recognised as a particular example (the case $\theta = 2$ in the
 notation of \cite{FW15}) of the Jacobi Muttalib--Borodin ensemble \cite{Mu95,Bor99}.
 Ref.~ \cite{FI17}  shows how the corresponding biorthogonal polynomials relate to the
 theory of the Selberg integrals, while a double integral form of the correlation kernel is
 given in \cite{FW15}.

 \subsection{Products with Induced Real Jacobi  Matrices}\label{sec:prod}
 
 Let  $X_i$ be drawn from the induced Jacobi ensemble of random real $2n \times 2n$ matrices,
with parameters $(\nu_j, \mu_j)$ 
as specified at the beginning of subsection~\ref{sec:Jac}. Let $A \in H$ be a member of a polynomial ensemble,
 and consider the random product matrix
 \begin{equation}\label{rp}
 X_M \cdots X_1 A X_1^T \cdots X_M^T.
 \end{equation}
 Iterative application of Corollary \ref{cor:jPDF.poly} tells us that Eq.~(\ref{rp}) is itself a polynomial ensemble.

 \begin{corollary}\label{corr3.4}
 Let $A$ be drawn form a polynomial ensemble with the jPDF for the singular values 
given by Eq.~(\ref{poly.def}) supported on $0 < a_j < 1$, and consider the random product matrices~(\ref{rp}). The singular values 
 $\{ a_j \}_{j=1}^n$ have the jPDF
  \begin{equation}\label{4.14}
 {C_n[w] \over n!} \prod_{j=1}^M {1 \over C_n(\nu_j, \mu_j)}  \Delta(a^2) \det
 \Big [ g_{j-1}^{(M)}(a_k)  \Theta({a_k (1 - a_k)}) \Big ]_{j,k=1}^n,
 \end{equation}
 where $C_n(\nu,\mu)$ is given by Eq.~(\ref{3.2aa})
  and $g_j^{(M)}$ is defined recursively according to
   \begin{equation}\label{4.15}
g_{j}^{(M)}(a)  = \int_a^1    {d {b} \over {b}} \, b^{2 \nu_j} (1 - b)^{2(\mu_j+n)}
g_{j-1}^{(M)} \Big ( {a \over b} \Big ).
\end{equation}
\end{corollary} 

The recurrence~(\ref{4.15}), with 
 \begin{equation}\label{4.15a}
2 \nu_j \mapsto \nu_j,  \qquad  2 (\mu_j +n )\mapsto \mu_j + n
\end{equation}
(where on the RHS of the second mapping $\mu_j:= m_j - 2n - \nu_j - 1$ in the notation of \cite{KKS15}),
 is known from earlier studies \cite{KKS15,Ku16} of complex random
product matrices $X_M \cdots X_1 A X_1^\dagger \cdots X_M^\dagger$ where each $X_i$ is drawn
from an induced complex Jacobi ensemble (truncated Haar distributed unitary matrices) and $A$ is
from a polynomial ensemble.

Perhaps the simplest choice of $A$ is from the ensemble specified in Corollary \ref{corr3.3}.
Then 
 \begin{equation}\label{4.16}
w_b(x) = x^{2 \nu+b - 1} (1 - x)^{2 \mu+ n + 1} \Theta(x(1-x)),
\end{equation}
and after relabelling
the product matrices, Eq.~(\ref{rp}) is equivalent to
 \begin{equation}\label{rp1}
 X_M \cdots X_1  \Big ( \eins_n \otimes \begin{bmatrix} 0 & i \\ -i & 0 \end{bmatrix} \Big ) X_1^T \cdots X_M^T,
 \end{equation}
 where each $X_i$ is drawn from the induced Jacobi ensemble of random $2n \times 2n$ real matrices
and parameters $(\nu_j, \mu_j)$. The relationship between the recurrence obtained in 
\cite{KKS15,Ku16}, and that implied by Eq.~(\ref{4.15}), gives us a corresponding relationship between
the jPDF~(\ref{rp1}) and a product ensemble involving matrices from the induced
complex Jacobi ensemble which is in agreement with~\ref{corr.lim}.

\begin{proposition}\label{prop4.5}
Let $Y_j$, $j=1,\dots,2M$ be drawn from the induced complex Jacobi ensemble with
probability densities proportional to 
\begin{equation}
(\det Y Y^\dagger)^{\nu_i - 1/2} \det (\eins_n - Y Y^\dagger)^{2\mu_i}
\Theta(\eins_n -  Y Y^\dagger) \ {\rm and}\ 
(\det Y Y^\dagger)^{\nu_i } \det (\eins_n - Y Y^\dagger)^{2\mu_i}
\Theta(\eins_n -  Y Y^\dagger),
\end{equation}
for $i=2j-1$ and $i=2j$, respectively ($j=1,\dots,M$), and consider the random product matrix
 \begin{equation}\label{YY}
  Y_{2M} \cdots Y_1  Y_1^\dagger \cdots Y_{2M}^\dagger.
  \end{equation}
  Then the jPDFs of the singular values $y\in A$ of Eq.~(\ref{YY}) and the singular values $x\in A$ of
  Eq.~(\ref{rp1}) are identical.
\end{proposition}

\begin{proof}
As noted the recurrence (\ref{4.15}), after the identification (\ref{4.15a}), is known from \cite{KKS15}. Moreover
with $g_1^{(M)}(x) = w_b(x)|_{\nu = \nu_1 \atop \mu = \mu_1}$, the solution of the recurrence can be
read off from formulas in  \cite[Eqs.~(2.25)--(2.27)]{KKS15}. The end result is that the linear span of $\{ g_j^{(M)}(x) \}$
consists of all functions of the form
 \begin{equation}\label{4.17}
{1 \over 2 \pi i} \int_C {q(s) \prod_{j=1}^{M} \Gamma(s + 2 \nu_j) \over
\prod_{j=1}^M \Gamma(s + 2\mu_j + 2 \nu_j + 2n +  1) } x^{-s} \, ds, \qquad 0 < x < 1,
 \end{equation}
where $q(s)$ is a polynomial of degree smaller than $n$, $C$ is a contour starting and ending at $-\infty$
and encircling the negative real axis.

Replacing $s$ by $2s$ and making use of the duplication formula for the gamma function, it
is also true that 
the linear span of $\{ g_j^{(M)}(x) \}$
consists of all functions of the form
 \begin{equation}\label{4.17b}
{1 \over 2 \pi i} \int_C {q(s) \prod_{j=1}^{M} \Gamma(s +  \nu_j) \Gamma(s +  \nu_j + 1/2) \over
\prod_{j=1}^M \Gamma(s + \mu_j +  \nu_j + n +  1/2)  \Gamma(s + \mu_j +  \nu_j + n +  1)} x^{-2s} \, ds, \qquad 0 < x < 1.
 \end{equation}
According to \cite[Eq.~(2.27)]{KKS15}, after replacing $s+1/2$ by $s$, this span in turn
is identical to $y$ times the
linear span of $\{\tilde{g}_j^{(2M)}(x^2)\}$, where $\Delta(x) \det [ \tilde{g}_j^{(2M)}(x_k) ]_{j,k=1}^n$ is up
to normalisation the jPDF of the singular values of the random product matrix~(\ref{YY}), hence 
\begin{equation}
\Delta(x^2) \det [  g_j^{(M)}(x_k) ]_{j,k=1}^M \propto \prod_{l=1}^n x_l \, \Delta(x^2) 
\det [ \tilde{g}_j^{(2M)}(x^2_k) ]_{j,k=1}^n.
\end{equation}
The left hand side is the jPDF of the singular values of the random matrix~(\ref{rp1}), while the right hand side is the 
 the jPDF of the squared eigenvalues of the matrix~(\ref{YY}), this establishing the 
stated result.
\end{proof}

This result is the analogue for products involving truncated orthogonal and unitary matrices of the result
\cite[Corollary 1.2]{FILZ17} relating the first of the random product matrices in
Eq.~(\ref{mn1}) to an jPDF of the singular values for products a complex Gaussian matrices;
recall the concluding paragraph of subsection~\ref{sec:Gin}. We remark that the corresponding bi-orthogonal
system follows from the results of \cite{KKS15}, as does the double contour integral form
of the correlation kernel, and are consistent
with the forms implied by Theorem~\ref{thm:biorth.poly}.

\section{Conclusions}\label{sec:conclusio}

We constructed a theoretical basis for dealing with the multiplicative convolution corresponding to the action of the general linear group $G={\rm Gl}_{\mathbb{R}}(2n)$ on the even dimensional real antisymmetric matrices $H={\rm o}(2n)$, i.e. $(g,x)\mapsto\,gxg^T$ with $g\in G$ and $x\in H$. This approach is based on harmonic analysis~\cite{He00} and follows the same ideas as already employed for various convolutions of the additive~\cite{KR16,Ki17} and the multiplicative~\cite{KK16a,KK16b} type. The only requirement to apply these ideas is that the probability density $P_H(x)$ has to be $K$-invariant, i.e. $P_H(x)=P_H(kxk^T)$ for all $x\in H$ and $k\in K$. For the situation of $x$ being either fixed or drawn from a polynomial ensemble and $g$ being drawn  from a factorizing ensemble, see Definition~\ref{def:poly.ens}, we were able to make the jPDF of the singular values of $y=gxg^T$, the corresponding bi-orthogonal system and the kernel explicit.

We want to emphasize, too, that the multiplication with rectangular matrices is implicitly covered by our results. Due to the $K$ invariance the spectral statistics of a $2n\times 2N$ dimensional rectangular real matrix $g$ is bijectively related to a square random matrix, see~\cite{IK14}. To reduce the general setting $gxg^T$ with $x$ of dimension $2N\times 2N$ to the situation considered in the present work, one has to distinguish two cases, either $N>n$ or $n>N$. When $n>N$ we can decompose
\begin{equation}
g=U\left[\begin{array}{c} g' \\ 0 \end{array}\right],\ U\in{\rm O}(2n)\ {\rm and}\ g'\in{\rm Gl}(2N),
\end{equation}
where the induced distribution of $g'$ inherits the $K$-invariance of $g$. In particular $gxg^T$ and $g'x{g'}^T$ share the same spectral statistics of their singular values apart from the generic zeros of $gxg^T$. The case $N>n$ is based on the decomposition $g=(g',0)U$ with $U\in{\rm O}(2N)$ and $g'\in{\rm Gl}(2n)$. Then we have to find first the jPDF of the singular values of the projected matrix $x'=\Pi_{2n,2N}UxU^T\Pi_{2n,2N}^T$ which can be recursively derived by applying Proposition~\ref{propW1} $N-n$ times. In the second step we can pursue the standard approach because $g'x'g'^T$ is exactly a product of the type we originally considered. In the particular cases of Section 4, these
constructions are intimately related  to the induced ensembles discussed therein \cite{FBKSZ12}.

An open question is how to generalize the ideas, we worked out here, to the setting of odd dimensional antisymmetric matrices. Obviously the spherical function has to be modified since the additional generic zero eigenvalue has to be taken care of. The theory for the multiplicative action of the complex general linear group on the Hermitian matrices is also non trivial. The open problem is again the corresponding spherical function. The spherical function is originally defined for positive definite Hermitian matrices~\cite{FK94,He00,KK16a}. How has the sign to be incorporated in the definition since it cannot be taken in the exponentiation otherwise we lose the analyticity of the exponents $s$. The situation is much simpler for the the quaternion general linear group ${\rm Gl}_{\mathbb{H}}(2n)$ and its action on the anti-self-dual matrices ${\rm usp}(2n)$ which is the Lie algebra of the unitary symplectic group ${\rm USp}(2n)$. We do not expect any complication in the latter case and the ideas pursued here should carry over easily.

\paragraph{Acknowledgements} 
We acknowledge support by the Australian Research Council through grant DP170102028 (PJF), 
the ARC Centre of Excellence for Mathematical and Statistical Frontiers (PJF,JRI,MK), and the German research council (DFG) via the CRC 1283: ``Taming uncertainty and profiting from randomness and low regularity in analysis, stochastics and their applications". Moreover MK wants to thank the University of Melbourne for its hospitality where this project was carried out.

\appendix

\section{Appendix}\label{AppA}

In this appendix we want to prove the explicit expression of the spherical function $\Phi$ as stated in Theorem~\ref{thm:spher.func}. We do this in three steps. First we consider the action of a corank $2$ projection on even dimensional antisymmetric matrices, see subsection~\ref{sec:proj}. We use this projection to construct a recurrence relation in the dimension for the group integral in the denominator of Eq.~\eqref{spher:as}, see subsection~\ref{sec:rec}, which is solved in subsection~\ref{sec:proof}.

\subsection{Eigenvalue PDF for a corank $2$ projection}\label{sec:proj}

For $C,X\in H$ both  $2n \times 2n$ real antisymmetric matrices, we define the matrix-valued
Fourier transform of $X$ by
\begin{equation}\label{X0}
\mathcal{F}P_H( C)=\int \exp\left[\frac{\imath}{2}\Tr X  C\right] P_H(X)dX
\end{equation}
Let the singular values of $C$ and $X$ be denoted $c\in A$ and $x\in A$,
respectively. Consider the circumstance that $\mathcal{F}P_H(C) = \mathcal{F}P_H(k C k^T)$ for all $k \in
K = {\rm O}(2n)$ and write $\hat{f}_X(c) = \mathcal{F}P_H( \imath c \otimes \tau_2)$.
With $p_A(x)$ denoting the jPDF of the singular values of $X$, by making use of the Harish-Chandra group
integral~ (\ref{5}) for $K={\rm O}(2n)$, it was shown in \cite[Prop.~3.4, with an extra factor of $1/n!$ for the
convention that the eigenvalues are not ordered]{FILZ17} that
\begin{equation}\label{w.2}
 p_A(x) = \prod_{j=0}^{n-1}\frac{(-1)^j2}{\pi(2j)!} \,
 \Delta_n(x^2) {1 \over n!} \int_{0}^\infty dc_1 \cdots  \int_{0}^\infty dc_n \,
 \hat{f}_X(c)
 \Delta_n(c^2) \prod_{j=1}^n \cos(x_j c_j).
 \end{equation}
 We will use Eq.~(\ref{w.2}), with $n \mapsto n - 1$, to deduce the jPDF for the corank 2 projection
 given by the $(2n - 2) \times (2n - 2)$ random matrix
 \begin{equation}\label{X1}
 X = \Pi_{2n-2,2n} k (\imath a \otimes \tau_2 ) k^T \Pi_{2n-2,2n}^T, \qquad
 k \in K = {\rm O}(2n),
 \end{equation}
 where $a = (a_1,\dots, a_n) \in A$ is fixed.
 
\begin{remark}
We want to underline that to be very rigorous we have to introduce a regularizing function in Eq.~\eqref{w.2} like a Gaussian or a function similar to Eq.~\eqref{xi.def} with a vanishing parameter $\epsilon\to0$ to guarantee the absolute integrability since $\hat{f}_X(c)\Delta_n(c^2)$ is not necessarily an $L^1$-function. In general it can grow like polynomial on $\mathbb{R}^n$. The particualr problem to be overcome is interchange of the inverse Laplace transform and the expectation value~\eqref{X0}. In the present case the function $\hat{f}_X(c) \Delta_n(c^2)$ is an $L^1$-function on $\mathbb{R}$, see below. Thus we can neglect these regularizing terms.
\end{remark}

   \begin{proposition}\label{propW1}
   Let $x = \{x_j \}_{j=1}^{n-1}\in A$ denote the singular values of the random matrix~(\ref{X1}), and let $p_A(x|a)$
   denote the corresponding jPDF. We have
  \begin{equation}\label{fE0}   
  p_A(x|a) = { (2n-2)!  \over (n-1)!} {\Delta_{n-1}(x^2) \over \Delta_n(a^2)} 
  \det   \left [ {  \displaystyle 1\quad\cdots\quad 1 \atop (a_k - x_j) \Theta ( a_k - x_j)
   } \right]_{j=1,\dots, n - 1 \atop k =1,\dots, n},
  \end{equation}
 where $\Theta$ denotes the Heaviside step function.  
    \end{proposition}

 \begin{proof}
 According to the definition (\ref{X0}), with $X$ given by Eq.~(\ref{X1}), upon use of the cyclic
 property of the trace and the definition of $\Pi_{2n-2,2n}$ we have
  \begin{equation}\label{fE}
  \mathcal{F}P_H(C) = \int_K \exp\left[- {1 \over 2} {\rm Tr} \, k (a \otimes \tau_2 ) k^T [ C \oplus {\rm diag} \, (0,0)]\right]
 d^*k.
  \end{equation}
 We denote with denoted $c=(c_1,\dots, c_{n-1})\in A$ the singular values of $C$. The matrix integral corresponding
  to the average (\ref{fE}) is an example of the Harish-Chandra group integral~(\ref{5}) for ${\rm O}(2n)$. 
  where $ C \oplus {\rm diag} \, (0,0) $ is to be thought of as $C \oplus \begin{bmatrix} 0 & \epsilon \\
  - \epsilon & 0 \end{bmatrix}$ with $\epsilon \to 0$.
  As such it has the evaluation
 \begin{equation}\label{fE1}
 \hat{f}_X(c) =\mathcal{F}P_H(\imath c\otimes\tau_2)=  \prod_{j=0}^{n-1}(2j)!
   \frac{(-1)^{n(n-1)/2}}{\Delta_n(a^2)\Delta_{n-1}(c^2) \prod_{l=1}^{n-1} c_l^2}
\det   \bigg [ {1\quad\cdots\quad 1 \atop
 \cos (c_j a_k)} \bigg ]_{j=1,\dots, n - 1 \atop k =1,\dots, n}.
\end{equation} 
Substituting Eq.~(\ref{fE1}) in Eq.~(\ref{w.2}), the latter with $n \mapsto n -1$, we obtain
\begin{equation}
p_A(x|a) = { (2n-2)!  \over (n-1)!} {\Delta_{n-1}(x^2) \over \Delta_n(a^2)}
\left(\prod_{j=1}^{n-1} \frac{-2}{\pi c_j^2}\right)\int_A d c
\det   \left[ {1\quad\cdots\quad 1 \atop
 \cos (c_j a_k)} \right ]_{j=1,\dots, n - 1 \atop k =1,\dots, n}
 \prod_{j=1}^{n-1} \cos x_j c_j.
 \end{equation}
 
 In preparation for further simplification, we subtract the first row of the determinant from each
 of the subsequent rows. The integrations can then be carried out row-by-row to give
\begin{equation}
 p_A(x|a) = { (2n-2)!  \over (n-1)!} {\Delta_{n-1}(x^2) \over \Delta_n(a^2)}
 \det   \left[ {  \displaystyle 1\quad\cdots\quad 1 \atop \displaystyle
{2 \over \pi}  \int_0^\infty {  \cos(x_j c) \, (1 - \cos (a_k c) ) \over c^2} \, dc} \right ]_{j=1,\dots, n - 1 \atop k =1,\dots, n}.
\end{equation}
Evaluating the integral with the help of the residue theorem gives
\begin{equation}
 {2 \over \pi}  \int_0^\infty {  \cos x_j c \, (1 - \cos a_k c ) \over c^2} \, dc =
  (a_k - x_j) \Theta ( a_k - x_j)
\end{equation}
 which concludes the proof.

 \end{proof}
 
 \begin{remark}\
 
 {\bf (a)} Let us order $a$ and $x$ as $a_1<a_2<\ldots<a_n$ and $x_1<x_2<\ldots<x_{n-1}$. Examining the determinant we see that in the intervals $[0,a_1]$ and $[a_{n-1},a_n]$ can be maximally occupied by a single singular value, namely $x_1$ and $x_{n-1}$, respectively. More singular values would yield that either the first three rows or last two rows become linearly dependent. Furthermore in each interval $[a_j,a_{j+1}]$ we cannot have more than two singular values $x_j$ because of the same reason (three or more rows become linearly dependent). The same also applies to the intervals $[x_j,x_{j+1}]$ which can maximally comprise two singular values of $a$ (three or more columns become linearly dependent when violated). A detailed analysis tells us that those ordering where the determinant does not vanish create a matrix which is upper triangular with maximally one additional lower diagonal.  Once a combined ordering of $a$ and $x$ is given as $0<\alpha_1<\alpha_2<\ldots<\alpha_{2n-2}<\alpha_{2n-1}$ with $\alpha$ a permutation of $(a,x)$, the determinant evaluates to $\prod_{j=1}^{n-1}(\alpha_{2j+1}-\alpha_{2j})$ which directly follows from the particular form of the matrix. The reason why $\alpha_1$ does not appear in the product follows from the following case discussion. We have either $x_1<a_1$ which means that we can subtract the first row times $x_1$ with the second row cancelling $x_1$ or it is $a_1<x_1$ where $a_1$ immediately drops out because of the Heaviside step-function which vanishes in this case.
 
 {\bf (b)} Let $b = {\rm diag}(b_1,\dots, b_n) \in \mathbb R^n$, $u \in K = {\rm U}(n)$, and consider the corank $1$
 projection given by the $(n-1) \times (n-1)$ random matrix $\Pi_{n-1,n} u b
 u^\dagger \Pi_{n-1,n}^T$. The method used proving Proposition \ref{propW1}, with the role
 of the Harish-Chandra group integral for ${\rm O}(2n)$ now played by its unitary (${\rm U}(n)$) counter part, can be adapted to show that the jPDF of the eigenvalues of this random matrix is
 equal to
\begin{equation}
 {\Delta_{n-1}(x) \over \Delta_n(b)}
 \det   \bigg [ {1\quad\cdots\quad 1 \atop \Theta(b_k - x_j) } \bigg ]_{j=1,\dots,n-1 \atop k=1,\dots,n} =
  {\Delta_{n-1}(x) \over \Delta_n(b)}
  \chi_{b_1 < x_1 < b_2 < \cdots < x_{n-1} < b_n},
\end{equation}
  where $\chi_J$ is the indicator function for $J$. This result is well known \cite{Ba01}.
 
 \end{remark}

 \subsection{A recurrence relation}\label{sec:rec}
 
We denote the numerator in Eq.~(\ref{spher:as}) by
 \begin{equation}\label{a:f.def}
f_n(s,x)=\int_{K={\rm O}(2n)}d^*k\prod_{j=1}^{n} [\det \Pi_{2j,2n}kxk^T\Pi_{2j,2n}^T]^{(s_j-s_{j+1})/2-1}.
\end{equation}
We first restrict to $s\in\mathbb{C}^n$ with ${\rm Re}(s_j-s_{j+1})\geq 2$ for all $j=1,\ldots, n-1$ to have a bounded integrand and then analytically continue at the end. Our present interest is in establishing a recurrence in the matrix dimension $n$.

\begin{lemma}[Recursion of $f_n$]\label{lem:rec}\

Let $n>2$, $a\in A=\mathbb{R}_+^n$ be the singular values of $x\in H={\rm o}(2n)$, assumed 
non-degenerate and $s\in\mathbb{C}^n$ with ${\rm Re}(s_j-s_{j+1})\geq 2$ for all $j=1,\ldots, n-1$. We have
\begin{multline}\label{recursion}
f_n(s,\imath a\otimes \tau_2)=\frac{(2n-2)!}{(n-1)!}(\det a)^{s_{n}+n-1} \int_{\mathbb{R}_+^{n-1}}d\tilde{a} \frac{\Delta_{n-1}(\tilde{a}^2)}{\Delta_{n}(a^2)}\det\left[\begin{array}{c} 1\quad\cdots\quad 1 \\ (a_c-\tilde{a}_b)\Theta(a_c-\tilde{a}_b) \end{array}\right]_{\substack{b=1,\ldots,n-1 \\ c=1,\ldots,n}}\\
\times f_{n-1}(\diag(s_{1}-s_{n}-n,\ldots,s_{n-1}-s_{n}-n),\imath\diag(\tilde{a}_1,\ldots\tilde{a}_{n-1})\otimes\tau_2).
\end{multline}
\end{lemma}

\begin{proof}
Since the function $f_n$ is $K$ invariant, i.e. $f_n(s,x)=f_n(s,kxk^T)$ for all $k\in K$ and $x\in H$, we can replace $x$ by $\imath a\otimes\tau_2$ with $a\in A$ its singular values. Then, the $j=n$
term in the product of the integrand evaluates to
\begin{equation}
( \det \Pi_{2n,2n} k x k^T \Pi_{2n,2n}^T)^{(s_n + n + 1)/2 - 1} = (\det a)^{s_n + n - 1}.
\end{equation}
Being independent of the orthogonal matrix $k$, this term can be taken outside of the integral. Since
$\Pi_{2j,2n}=\Pi_{2j,2(n-1)}\Pi_{2(n-1),2n}$ for each $j=1,\ldots,n-1$, for the remaining terms
we can write
\begin{equation}
\begin{split}
&\prod_{j=1}^{n-1} (\det \Pi_{2j,2n}kxk^T\Pi_{2j,2n}^T )^{(s_n + n + 1)/2 - 1} \\
=&\prod_{j=1}^{n-1} (\det \Pi_{2j,2(n-1)}\Pi_{2(n-1),2n}kxk^T\Pi_{2(n-1),2n}^T\Pi_{2j,2(n-1)}^T )^{(s_n + n + 1)/2 - 1}.
\end{split}
\end{equation}
We then introduce an orthogonal matrix $\tilde{k}\in{\rm O}(2n-2)$ by the multiplicative shift $k\to\diag(\tilde{k},\eins_2)k$ and integrate $\tilde{k}$ via the Haar measure on ${\rm O}(2n-2)$ to obtain
\begin{equation}
\begin{split}
f_n(s,\imath a\otimes \tau_2)=&(\det a)^{s_{n}+n-1}\int_{K={\rm O}(2n)}d^*k\int_{K={\rm O}(2n-2)}d^*\tilde{k}\\
&\times\prod_{j=1}^{n-1} [\det \Pi_{2j,2(n-1)}\tilde{k}\Pi_{2(n-1),2n}k(\imath a\otimes \tau_2)k^T\Pi_{2(n-1),2n}^T\tilde{k}^T\Pi_{2j,2(n-1)}^T]^{(s_j-s_{j+1})/2-1}\\
=&(\det a)^{s_{n}+n-1}\int_{K={\rm O}(2n)}d^*k \\
&\times f_{n-1}(\diag(s_{1}-s_{n}-n),\ldots,s_{n-1}-s_{n}-n),\Pi_{2(n-1),2n}k(\imath a\otimes \tau_2)k^T\Pi_{2(n-1),2n}^T).
\end{split}
\end{equation}
Since also $f_{n-1}$ is ${\rm O}(2n-2)$-invariant we can replace $\tilde{x}=\Pi_{2(n-1),2n}k(\imath a\otimes \tau_2)k^T\Pi_{2(n-1),2n}^T$ by $\imath\tilde{a}\otimes\tau_2$ with $\tilde{a}\in\mathbb{R}_+^{n-1}$ the singular values of the $2(n-1)\times2(n-1)$ dimensional real antisymmetric random matrix $\Pi_{2(n-1),2n}k(\imath a\otimes \tau_2)k^T\Pi_{2(n-1),2n}^T$. The jPDF of $\tilde{a}$ is given in Proposition \ref{propW1} which yields the recursion~(\ref{recursion}).
\end{proof}

\subsection{Proof of Theorem~\ref{thm:spher.func}}\label{sec:proof}

Let  $a\in A=\mathbb{R}_+^n$ be the singular values of $x\in H={\rm o}(2n)$, assumed non-degenerate, and $s\in\mathbb{C}^n$ with ${\rm Re}(s_j-s_{j+1})\geq 2$ for all $j=1,\ldots, n-1$.
The spherical function~\eqref{spher:as} is given in terms of $f_n$ according to 
\begin{equation}\label{a.proof.2}
\Phi(s;x)=\frac{f_n(s,x)}{f_n(s,\imath\eins_n\otimes\tau_2)},
\end{equation}
and so the explicit knowledge of $f_n$ implies the explicit form of $\Phi$.

Regarding the latter, we will use complete induction to show
\begin{equation}\label{a.proof.3}
f_n(s; \imath a\otimes\tau_2)=c_n(s)\frac{\det[a_c^{s_b+n-1}]_{b,c=1,\ldots,n}}{\Delta_n(a^2)},
\end{equation}
where
\begin{equation}\label{cns}
c_n(s) = \frac{(-1)^{n(n-1)/2}\prod_{j=0}^{n-1}(2j)!}{\Delta_n(s) \prod_{1\leq k<l\leq n}(s_k-s_l-1)}.
\end{equation}
For $n=1$, we have $f_1(s_1;\imath a_1\tau_2)=a_1^{s_1}$ agreeing with the ansatz~\eqref{a.proof.3} with $c_1(s_1)=1$ as is consistent with Eq.~(\ref{cns}).

For the induction step we substitute the ansatz~\eqref{a.proof.3}  for $f_{n-1}$ 
in the recursion \eqref{recursion} to obtain
\begin{multline}\label{a.proof.4}
f_n(s,\imath a\otimes \tau_2)=c_{n-1}(\diag(s_{1}-s_{n}-n),\ldots,s_{n-1}-s_{n}-n))\frac{(2n-2)!}{(n-1)!}(\det a)^{s_{n}+n-1}\\
\times \int_{\mathbb{R}_+^{n-1}}d\tilde{a} \frac{\Delta_{n-1}(\tilde{a}^2)}{\Delta_{n}(a^2)}\det\left[\begin{array}{c} 1\quad\cdots\quad 1 \\ (a_c-\tilde{a}_b)\Theta(a_c-\tilde{a}_b) \end{array}\right]_{\substack{b=1,\ldots,n-1 \\ c=1,\ldots,n}}\frac{\det[\tilde{a}_c^{s_b-s_n-2}]_{b,c=1,\ldots,n-1}}{\Delta_{n-1}(\tilde{a}^2)} \\
=c_{n-1}(\diag(s_{1}-s_{n}-n),\ldots,s_{n-1}-s_{n}-n))(2n-2)!\frac{(\det a)^{s_{n}+n-1}}{\Delta_{n}(a^2)}\\
\times \det\left[\begin{array}{c} 1\quad\cdots\quad 1 \\ \int_0^{a_c} dt(a_c-t)t^{s_b-s_n-2} \end{array}\right]_{\substack{b=1,\ldots,n-1 \\ c=1,\ldots,n}}\\
=\frac{(-1)^{n-1}(2n-2)!}{\prod_{l=1}^{n-1}(s_l-s_n-1)(s_l-s_n)}c_{n-1}(\diag(s_{1}-s_{n}-n),\ldots,s_{n-1}-s_{n}-n))\frac{\det[a_c^{s_b+n-1}]_{b,c=1,\ldots,n}}{\Delta_n(a^2)}.
\end{multline}
In the second equality we used a variant of Andr\'eief's integration identity~\cite[Appendix C.1]{Kieburg:2010}, and in the third we integrated by parts using the fact that the boundary terms vanish for  $b=1,\ldots,n-1$ because ${\rm Re}(s_b-s_{n})\geq 2(n-b)$. Comparison of this result with the ansatz~\eqref{a.proof.3} yields the recursion in the constants
\begin{equation}
c_n(s)=\frac{(-1)^{n-1}(2n-2)!}{\prod_{l=1}^{n-1}(s_l-s_n-1)(s_l-s_n)}c_{n-1}(\diag(s_{1}-s_{n}-n),\ldots,s_{n-1}-s_{n}-n)),
 \end{equation}
which with the initial condition $c_1(s)=1$ implies Eq.~(\ref{cns}).

For $a=\eins_n$ we obtain with the help of l'H\^opital's rule
\begin{equation}\label{a.proof.5}
f_n(s; \imath \eins_n\otimes\tau_2)=\frac{\prod_{j=0}^{n-1}(2j)!/j!}{\prod_{1\leq k<l\leq n}2(s_k-s_l-1)}.
\end{equation}
Dividing Eq.~\eqref{a.proof.4} by Eq.~\eqref{a.proof.5} yields Eq.~\eqref{spher:as.b}.

The result~\eqref{a.proof.4} can be analytically continued in the complex variables $z_j=s_j-s_{j+1}-2$ with $j=1,\ldots,n-1$ when dividing the spherical function $\Phi(s,x)$ by $\max_{j=1,\ldots,n}\{a_j^{\sum_{l=1}^ns_l}\}$. This division is equivalent with restricting all $a_j$ to the interval $[0,1]$. Then the singular values of each matrix $\Pi_{2j,2n}kxk^T\Pi_{2j,2n}^T$ are also restricted to the interval $[0,1]$ and $f_n(s,\imath a\otimes \tau_2)\prod_{1\leq k<l\leq n}(s_k-s_l-1)(s_l-s_k)$ is a bounded analytic function in each $z_j$ on the complex half-hyperplanes ${\rm Re}\, z_j\geq0$. Carlson's theorem can be applied to Eq.~\eqref{a.proof.4} for each $z_j$ to uniquely analytically continue to the whole complex plane which concludes the proof of Theorem~\ref{thm:spher.func}.

 \end{document}